\documentclass[10pt,oneside,reqno]{amsart}

\makeatletter
\@namedef{subjclassname@2020}{\textup{2020} Mathematics Subject Classification}
\makeatother

\usepackage[top=2.5 cm,bottom=2 cm,left=2 cm,right=3 cm]{geometry}
\usepackage[utf8]{inputenc}
\usepackage{color}
\usepackage{amssymb}

\usepackage[english]{babel}

\usepackage[backref=page]{hyperref}

\usepackage{esint}
\usepackage{graphicx}
\usepackage{hyperref}
\usepackage{bm}
\usepackage{xcolor}

\usepackage{enumerate}

\newcommand*{\mailto}[1]{\href{mailto:#1}{\nolinkurl{#1}}}

\numberwithin{equation}{section}

\newtheorem{example}{Example}[section]
\newtheorem{definition}[example]{Definition}
\newtheorem{theorem}[example]{Theorem}
 \newtheorem{proposition}{Proposition}
\newtheorem{lemma}{Lemma}
 
\newtheorem{remark}[example]{Remark}
\newtheorem*{maintheorem*}{Main Theorem}
\allowdisplaybreaks
\numberwithin{equation}{section}

\renewcommand{\i}{\ifmmode\mathit{\mathchar"7010 }\else\char"10 \fi}
\renewcommand{\j}{\ifmmode\mathit{\mathchar"7011 }\else\char"11 \fi}
\newcommand{\R}{\mathbb{R}}

\newcommand{\px}{\partial_x}
\newcommand{\pt}{\partial_t}
\newcommand{\pxi}{\partial_\xi}
\newcommand{\peta}{\partial_\eta}

\newcommand{\ve}{\varepsilon}

{%

\begin{enumerate}}%
{\end{enumerate}}

{%

\begin{enumerate}}%
{\end{enumerate}}

\begin{document}

\title[Semigroup of conservative solutions of the two-component equation]
{Global semigroup of conservative weak 
solutions of the two-component Novikov equation}

 \author[Karlsen]{K. H. Karlsen}
\author[Rybalko]{Ya. Rybalko}

\address[Kenneth Hvistendahl Karlsen]{\newline
  Department of Mathematics, \newline University of Oslo, \newline
  PO Box 1053, Blindern -- 0316 Oslo, Norway}
\email[]{kennethk@math.uio.no}

\address[Yan Rybalko]{\newline
	Department of Mathematics, \newline University of Oslo, \newline
	PO Box 1053, Blindern -- 0316 Oslo, Norway
	\medskip
	\newline
	Mathematical Division, 
	\newline B.Verkin Institute for Low Temperature Physics and Engineering
	of the National Academy of Sciences of Ukraine,
	\newline 47 Nauky Ave., Kharkiv, 61103, Ukraine}
\email[]{rybalkoyan@gmail.com}



\thanks{We deeply appreciate the insightful and constructive feedback provided by an anonymous referee, which has significantly enhanced the clarity and overall quality of our paper.
The work of Yan Rybalko was partially supported 
by the N.I. Akhiezer Foundation and
the European Union’s Horizon Europe research and innovation 
programme under the Marie Sk\l{}odowska-Curie grant No 101058830.
The work of Kenneth H. Karlsen was partially funded by the 
Research Council of Norway under project 351123 (NASTRAN)}

\date{\today}

\begin{abstract}
We study the Cauchy problem for the two-component Novikov system 
with initial data $u_0, v_0$ in $H^1(\mathbb{R})$ such that the product 
$(\partial_x u_0)\partial_x v_0$ belongs to $L^2(\mathbb{R})$. 
We construct a global semigroup of conservative weak solutions. 
We also discuss the potential concentration phenomena of 
$(\partial_x u)^2dx$, $(\partial_x v)^2dx$, and 
$\left((\partial_x u)^2(\partial_x v)^2\right)dx$, which contribute 
to wave-breaking and may occur for a set of time 
with nonzero measure. Finally, we establish the 
continuity of the data-to-solution 
map in the uniform norm.
\end{abstract}

\maketitle

\numberwithin{lemma}{section}
\numberwithin{proposition}{section}

\section{Introduction}

In this paper we consider the following 
two-component Novikov system \cite{L19}:
\begin{align}
	\nonumber
	\label{t-c-N}
	&\partial_tm+\px(uvm)+(\px u)vm=0,
	&& m=m(t,x),\,u=u(t,x),\,v=v(t,x),\\
	&\partial_tn+\px(uvn)+u(\px v)n=0,
	&& n=n(t,x),\\
	\nonumber
	&m=u-\partial_{x}^2u,\,\,n=v-\partial_{x}^2v,
	&& u,v\in\mathbb{R},\quad
	t,x\in\mathbb{R},
\end{align}
where we extend the temporal variable $t$ over the real line 
$\mathbb{R}$ to facilitate the two-place reduction 
$v(t, x) = u(-t, -x)$ of \eqref{t-c-N}, which inherently involves 
both $t$ and $-t$ (see \eqref{nN} below). Further details 
regarding two-place systems 
are provided later in the introduction.

We assume that the initial data $(u,v)(0,x)=(u_0,v_0)$ 
for the system \eqref{t-c-N} belongs to the metric space 
$\left(\Sigma, d_\Sigma\right)$, where
\begin{equation}\label{Sig}
	\Sigma=
	\left\{
	(f,g): f,g\in H^1(\mathbb{R})\mbox{ and }
	\int_{-\infty}^{\infty}
	(\px f)^2(\px g)^2\,dx<\infty
	\right\},
\end{equation}
and
\begin{equation}\label{d^2}
	d^2_\Sigma((f_1,g_1), (f_2,g_2))
	=\|f_1-f_2\|^2_{H^1}+\|g_1-g_2\|^2_{H^1}
	+\|(\px f_1)\px g_1-(\px f_2)\px g_2\|^2_{L^2}.
\end{equation}
We have
\begin{equation}\label{eq:inclusion-intro}
	\left(H^1\cap W^{1,4}\right)\times
	\left(H^1\cap W^{1,4}\right)\subset\Sigma,
\end{equation}
and the function $d_\Sigma(\cdot,\cdot)$ satisfies 
the axioms of a metric. Note that $\Sigma$ is not 
a linear space, but $\left(\Sigma, d_\Sigma\right)$ is 
a complete metric space (see Proposition \ref{Sm}).

When setting $v = u$, the two-component 
system \eqref{t-c-N} reduces to the 
scalar Novikov equation:
\begin{equation}\label{N}
	\partial_tm+\px(u^2m)+(\px u)um=0.
\end{equation}
This equation was derived using the perturbative 
symmetry approach \cite{N09} applied to 
the generalized Camassa-Holm equation:
$$
(1-\px^2)\pt u=F(u,\px u, \px^2u,\px^3 u,\dots),
\quad u=u(t,x),
$$
which is known for possessing an infinite 
hierarchy of higher symmetries. It was also 
formulated as a shallow-water wave model under 
a suitable approximation of the Euler equation 
for an incompressible fluid \cite{CHL22}. 
The Novikov equation \eqref{N} is recognized 
as an integrable system. It features a Lax pair, 
a bi-Hamiltonian structure \cite{HW08,N09}, 
and supports both peakon \cite{HLS09} and smooth 
soliton solutions \cite{M13}. Moreover, it can be linearized 
through the inverse scattering transform method \cite{BSZ16}. 
For an in-depth discussion of the stability of 
peakons under small perturbations, we 
direct the reader to \cite{CDL23,CP21}.

The solution to the Cauchy problem for the 
Novikov equation can develop finite-time singularities in 
the form of wave breaking \cite{JN12} (cf.~blow-up phenomena 
for the Camassa-Holm equation \cite{CE98}). Under additional assumptions 
on the sign of the initial data $u_0$ and its derivatives, unique global 
weak solutions in the Sobolev spaces $H^1 \cap W^{1,\infty}$ and 
$H^1 \cap W^{1,4}$ were established in \cite{WG16} and \cite{L13, ZYG23}, 
respectively. By applying the method of characteristics proposed 
by Bressan and Constantin \cite{BC07} (see also 
\cite{BCZ15, BC07d}), the authors in \cite{CCL18} obtained 
a unique global conservative solution for the Novikov equation, 
potentially featuring finite-time singularities, for general initial data 
in $H^1 \cap W^{1,4}$. A Lipschitz metric in this space was 
constructed in \cite{CCCS18}. Additionally, a global semigroup 
of dissipative solutions of \eqref{N} in $H^1 \cap W^{1,4}$ 
was derived in \cite{ZYM18}.

Another notable reduction of \eqref{t-c-N} can 
be derived by setting $v(t, x) = u(-t, -x)$, leading to 
the following two-place (nonlocal) 
form of the Novikov equation:
\begin{equation}\label{nN}
	\partial_tm(t,x)
	+\px(u(t,x)u(-t,-x)m(t,x))
	+(\px u(t,x))u(-t,-x)m(t,x)=0,\quad
	t,x\in\mathbb{R}.
\end{equation}
Ablowitz and Musslimani introduced a class of integrable 
two-place equations in \cite{AM13}. They defined a \textit{nonlocal}
nonlinear Schr\"odinger (NNLS) equation, which
 is a reduction of the Ablowitz-Kaup-Newell-Segur system:
\begin{equation}
	\label{NNLS}
	\mathrm{i}\partial_{t}q(t,x)+\partial_{x}^2q(t,x)
	+2\sigma q^{2}(t,x)\overline{q}(t,-x)=0,
	\quad 
	\mathrm{i}^2=-1,\quad\sigma=\pm1,
\end{equation}
where $\overline{q}$ is a complex conjugate of $q$
(notice that \eqref{NNLS} has a nonlocality with 
respect to $x$ only, cf.\,\,\eqref{nN}).
Equation \eqref{NNLS} with $\sigma=1$ has a 
one-soliton solution \cite{AM13}
\begin{equation*}
	q_{\alpha}(t,x)=
	\frac{3\alpha}
	{e^{-4\mathrm{i}\alpha^2t
	+2\alpha x}
	+e^{-\mathrm{i}\alpha^2t
		-\alpha x}},
	\quad\alpha>0,
	\quad (t,x)\neq\left(
	\frac{(2n+1)\pi}{3\alpha^2},0
	\right),\,n\in\mathbb{Z},
\end{equation*}
which blows-up in finite time in the $L^\infty$ norm.
The global well-posedness questions of the NNLS 
equation were investigated in \cite{CLW13, RS23, ZF24}.

Apart from being integrable, the NNLS equation 
satisfies a parity-time (PT) symmetry property, that is, $q(t,x)$ 
and $\overline{q}(-t,-x)$ satisfy \eqref{NNLS} simultaneously.
Notice that two-place Novikov equation \eqref{nN}, proposed here, 
is also PT-symmetric: both $u(t,x)$ and $u(-t,-x)$ 
are solutions of \eqref{nN}.
Since the dynamics of two-place equations depends on the 
values of the solution at non-adjacent states, such equations 
can be useful for describing phenomena having strong correlations 
and entanglement of the events at different places \cite{LH17}.
For the other examples of the nonlocal/two-place equations 
we refer to, e.g., \cite{AM17, F16, LQ17}.

For the two-component Novikov system, \cite{MM20} 
established local well-posedness of \eqref{t-c-N} in certain Besov 
spaces (see also \cite{ZQ21}) and showed that blow-up occurs 
only through wave breaking. Multipeakons were studied via 
inverse spectral methods in \cite{CS23}, while the stability 
of one-peakon solutions was addressed in \cite{HLQ23, HLLQ24}. 
A unique global weak solution of \eqref{t-c-N} was 
obtained in \cite{HQ20} under sign assumptions 
on the initial data $m_0$ and $n_0$.

Using the method of characteristics and adapting 
the approach of Bressan and Constantin \cite{BC07}, 
the work \cite{HQ21} constructs a global conservative 
weak solution $(u,v)(t,\cdot)\in\Sigma$ 
(see \eqref{Sig}) of \eqref{t-c-N}, subject to the initial 
data $(u_0,v_0)\in\left(H^1\cap W^{1,4}\right)\times\left(H^1
\cap W^{1,4}\right)$, without additional sign assumptions 
on the initial data \cite[Theorem 1.1]{HQ21}; see also \cite{HQ20}. 
Moreover, it is shown that the components $u$ and $v$ of the 
obtained solution are H\"older continuous in $x$ and $t$ with 
exponent $1/2$, and that the measure 
$\left((\px u)^2(\px v)^2\right)dx$ can concentrate 
in finite time \cite[Theorem 1.1]{HQ21}.

Recalling \eqref{eq:inclusion-intro}, it follows that the 
global solutions constructed in \cite{HQ21} 
generally fail to preserve regularity. Consequently, establishing 
a semigroup of such solutions is unattainable within this framework. 
The main objective of the present work is to address this 
limitation by constructing a global semigroup of conservative 
weak solutions to \eqref{t-c-N} in the metric 
space $(\Sigma, d_{\Sigma})$ defined by 
\eqref{Sig} and \eqref{d^2}.

To achieve this, we revisit the method of 
characteristics for the two-component Novikov equation, originally 
developed in \cite{HQ21}, extending its applicability to 
general initial data $(u_0, v_0) \in \Sigma$. Although 
$\Sigma$ is not a vector space, we demonstrate that the 
associated semilinear ODE system, described in \eqref{ODE}--\eqref{id} 
below, can still be formulated in the same Banach space $E$ 
as in \cite{HQ21}; see Section \ref{Gsl}. Following similar 
arguments as those in \cite[Sections III, IV]{HQ21}, we establish 
the global well-posedness of the ODE system 
and thereby construct a global conservative 
weak solution of \eqref{t-c-N} in $\Sigma$.


To construct a semigroup of global solutions (see Section \ref{sgs}), 
we introduce the positive Radon measure $\mu_{(t)}$, 
whose absolutely continuous part is given by
$$
d\mu_{(t)}^{ac} = \left(
(\px u)^2 + (\px v)^2 + (\px u)^2 (\px v)^2\right) dx.
$$
This measure encodes the energy distribution of 
the solution $(u,v)$, particularly at instances of 
wave-breaking. It allows us to distinguish between 
different solutions when the terms $(\px u)^2\,dx$, 
$(\px v)^2\,dx$, or $\left((\px u)^2(\px v)^2\right)dx$ 
become concentrated.  We then show that the triple 
$\left(u(t),v(t),\mu_{(t)}\right)$ satisfies the semigroup 
property, enabling the unique continuation of the solution $(u,v)$ 
beyond collision times while preserving the conserved 
quantities $E_{u_0}$, $E_{v_0}$, $G_0$, and $H_0$ 
(see \eqref{consq} below).

Notice that wave-breaking in the two-component system \eqref{t-c-N} 
can occur due to the concentration of either $(\px u)^2\,dx$, 
$(\px v)^2\,dx$, or $\left((\px u)^2(\px v)^2\right)dx$ in finite time. 
This behavior contrasts sharply with the Novikov equation, where 
only the concentration of $(\px u)^4\,dx$ is possible, while the 
$H^1$ norm of the solution persists for all times, as 
discussed in \cite{CCL18,CCCS18}. A similar energy 
concentration mechanism occurs in the Camassa-Holm equation, 
where $(\px u)^2\,dx$ can concentrate \cite{BC07,BC07d}. 
The difference in energy concentration properties between the 
Novikov equation and the two-component system arises from 
their respective ``transport coefficients", $u^2$ and $uv$. 
Indeed, for the $N$-peakon solution of the 
Novikov equation, given by 
$$
u(t,x)=\sum\limits_{i=1}^{N}p_i(t)e^{-|x-q_i(t)|},
$$
the evolution of the peakon positions $q_i$ 
follows the ODE system:
$$
\frac{d}{dt}q_i=u^2(t,q_i), \quad i=1,\dots, N,
$$
as shown in \cite{HW08}. This implies that all peakons move 
to the right, allowing only over-taking collisions.
In contrast, considering the corresponding $N$-peakon 
solutions for the two-component system \eqref{t-c-N}, where 
$$
u(t,x)=\sum\limits_{i=1}^{N}p_i(t)e^{-|x-q_i(t)|}, 
\quad v(t,x)=\sum\limits_{i=1}^{N}r_i(t)e^{-|x-q_i(t)|},
$$
the evolution of $q_i$ is governed by the system of ODEs:
$$
\frac{d}{dt}q_i=(uv)(t,q_i), \quad i=1,\dots, N,
$$
as established in \cite{CS23}. This allows peakons to 
travel in either direction, making both head-on and over-taking 
collisions possible, similar to the behavior 
observed in the Camassa-Holm equation.

Last but not least, we establish a new 
continuity property of the data-to-solution map 
for the two-component Novikov system \eqref{t-c-N}. 
Assuming that the initial data 
$(u_{0,n},v_{0,n})$ converge to $(u_0,v_0)$ 
in the metric space $(\Sigma,d_\Sigma)$ 
defined by \eqref{Sig} and \eqref{d^2}, we can conclude 
that the corresponding solution components $u_n$ and $v_n$ 
converge uniformly in $t$ and $x$ 
to $u$ and $v$, respectively, 
where $(u,v)$ solves the two-component 
Novikov system with initial data $(u_0,v_0)$.

The article is organized as follows: Section \ref{mr} 
introduces the essential notations and provides a 
summary of the core results presented in this paper.
In Section \ref{chv} we 
derive the ODE system for the new variables defined in terms of the solution $(u,v)$ of \eqref{t-c-N} along the characteristic.
Sections \ref{Gsl} and \ref{GS} consider the local 
and global well-posedness, respectively, of the 
associated ODE system. In Section \ref{GsN}, we 
construct a global conservative weak solution for the 
two-component system by utilizing a unique global 
solution of the ODE system, and we demonstrate 
the continuous dependence of the solution 
on the initial data with respect to the uniform norm. 
Finally, Section \ref{sgs} defines a global 
semigroup of conservative solutions.

\section{Main results}\label{mr}

For the initial data $(u_0,v_0)\in\Sigma$,
we introduce the following quantities:
\begin{equation}\label{consq}
	\begin{split}
		&E_{u_0}=\int_{-\infty}^{\infty}
		\left(u_0^2+(\px u_0)^2\right)(x)\,dx,\quad
		E_{v_0}=\int_{-\infty}^{\infty}
		\left(v_0^2+(\px v_0)^2\right)(x)\,dx,\\
		&G_{0}=\int_{-\infty}^{\infty}
		\left(u_0v_0+(\px u_0)\px v_0\right)(x)\,dx,\\
		&H_0=\int_{-\infty}^{\infty}
		\left(3u_0^2v_0^2+u_0^2(\px v_0)^2+(\px u_0)^2v_0^2
		+4u_0(\px u_0)v_0\px v_0-(\px u_0)^2(\px v_0)^2\right)
		(x)\,dx,
	\end{split}
\end{equation}
as well as
\begin{equation}\label{Ku}
	K_{u_0}=\frac{1}{4}\left(E_{u_0}
	(7E_{u_0}E_{v_0}-H_0)\right)^{1/2}, \quad
	K_{v_0}=\frac{1}{4}\left(E_{v_0}
	(7E_{u_0}E_{v_0}-H_0)\right)^{1/2}.
\end{equation}
Observe that $7E_{u_0}E_{v_0}-H_0\geq 0$, which can 
be derived from \eqref{u_x^2}, as shown below, with 
$u_0$ and $v_0$ substituted 
for $u$ and $v$, respectively.

Applying the operator $(1-\px^2)^{-1}$ to both sides of \eqref{t-c-N}, 
we obtain the following nonlocal system:
\begin{align}\label{t-c-N-n}
\begin{split}
&\pt u+uv\px u+\px P_1+P_2=0,\quad 
u=u(t,x),\,\,v=v(t,x),\,\, P_j=P_j(t,x),\,\,j=1,2,\\
&\pt v+uv\px v+\px S_1+S_2=0,\quad
S_j=S_j(t,x),\,\,j=1,2,
\end{split}
\\
\label{iid}
&u(0,x)=u_0(x),\quad v(0,x)=v_0(x),
\end{align}
where
\begin{equation}\label{P12}
	\begin{split}
		&P_1(t,x)=(1-\px^2)^{-1}\left(
		u^2v+u(\px u)\px v+\frac{1}{2}v(\px u)^2
		\right)(t,x),\\
		&P_2(t,x)=\frac{1}{2}(1-\px^2)^{-1}
		\left((\px u)^2\px v\right)(t,x),
	\end{split}
\end{equation}
and
\begin{equation}
	\label{S12}
	\begin{split}
		&S_1(t,x)=(1-\px^2)^{-1}\left(
		uv^2+v(\px u) \px v+\frac{1}{2}u(\px v)^2
		\right)(t,x),\\
		&S_2(t,x)=\frac{1}{2}(1-\px^2)^{-1}
		\left((\px u)(\px v)^2\right)(t,x).
	\end{split}
\end{equation}
Taking $u=v$, equation \eqref{t-c-N-n} transforms into 
the Novikov equation in its nonlocal form \cite{CCL18}.

To motivate the definition of a global weak 
solution of the Cauchy problem 
\eqref{t-c-N-n}--\eqref{iid}, consider a 
test function $\phi_u\in C^\infty([-T,T]\times\mathbb{R})$ 
with compact support. By multiplying the first equation 
in \eqref{t-c-N-n} by $\px\phi_u$ and applying integration 
by parts, we derive the following expression:
\begin{equation*}
	\begin{split}
	&\int_{-T}^T\int_{-\infty}^{\infty}
	\left((\pt u)\px\phi_u
	+uv(\px u)\px\phi_u
	+(\px P_1+P_2)\px\phi_u\right)
	\,dx\,dt\\
	&=\int_{-T}^T\int_{-\infty}^{\infty}
	\left((\px u)(\pt\phi_u+uv\px\phi_u)
	+\left(u^2v+u(\px u)\px v+\frac{1}{2}v(\px u)^2
	-P_1-\px P_2\right)\phi_u\right)
	\,dx\,dt.
	\end{split}
\end{equation*}
A similar equation can be derived for $v$ using 
the second equation in \eqref{t-c-N-n}.
This motivates us to give 
the following definition:

\begin{definition}[Global weak solution 
of \eqref{t-c-N-n}--\eqref{iid}]\label{defs}
Suppose that $(u_0,v_0)\in\Sigma$, where 
$\Sigma$ is defined in \eqref{Sig}.
We say that a vector function $(u,v)(t,x)$ is a 
global weak solution of the Cauchy 
problem \eqref{t-c-N-n}-\eqref{iid} 
on $\mathbb{R}$, if
$(u,v)(0,x)=(u_0,v_0)(x)$ for all 
$x\in\mathbb{R}$ and $(u,v)$ satisfies 
the equations
\begin{equation}\label{uvw}
	\begin{split}
	&\int_{-T}^T\int_{-\infty}^{\infty}
	\biggl((\px u)(\pt\phi_u+uv\px\phi_u)\\
	&\left.\qquad\qquad\,\,\,\,
	+\left(u^2v+u(\px u)\px v+\frac{1}{2}v(\px u)^2
	-P_1-\px P_2\right)\phi_u\right)\,dx\,dt=0,\\
	&\int_{-T}^T\int_{-\infty}^{\infty}
	\biggl((\px v)(\pt\phi_v+uv\px\phi_v)\\
	&\left.\qquad\qquad\,\,\,\,+\left(uv^2+v(\px u)\px v
	+\frac{1}{2}u(\px v)^2-S_1-\px S_2\right)\phi_v\right)\,dx\,dt=0,
	\end{split}
\end{equation}
for all test functions 
$\phi_u,\phi_v\in C^\infty((-T,T)\times\mathbb{R})$ 
with compact support and arbitrary $T>0$.
Moreover, $u$ and $v$ have the following properties:
\begin{enumerate}
	\item $(u,v)(t,\cdot)$ belongs to 
	$\Sigma$
	for any fixed $t\in\mathbb{R}$;

	\item $u(t,\cdot)$, $v(t,\cdot)$ are Lipschitz 
	continuous with values in $L^2$, that is, for 
	all $t_1,t_2\in[-T,T]$, for any fixed $T>0$, 
	the functions $u$ and $v$ satisfy
	\begin{equation}\label{uL}
		\|u(t_1,\cdot)-u(t_2,\cdot)\|_{L^2},
		\|v(t_1,\cdot)-v(t_2,\cdot)\|_{L^2}
		\leq C|t_1-t_2|,
	\end{equation}
	for some $C=C(E_{u_0},E_{v_0},H_0,T)>0$, 
	where $E_{u_0},E_{v_0}$ 
	and $H_0$ are defined in \eqref{consq};

	\item $u(t,x)$ and $v(t,x)$ are H\"older continuous 
	on $[-T,T]\times\mathbb{R}$ with exponent $1/2$ for any 
	fixed $T>0$, that is, for all $t_1,t_2,\in[-T,T]$ 
	and $x_1,x_2\in\mathbb{R}$ we have
	\begin{equation}\label{uH}
		|u(t_1,x_1)-u(t_2,x_2)|,
		|v(t_1,x_1)-v(t_2,x_2)|\leq
		C\left(|t_1-t_2|^{1/2}
		+|x_1-x_2|^{1/2}\right),
	\end{equation}
for some $C=C(E_{u_0},E_{v_0},H_0,T)>0$.
\end{enumerate}
\end{definition}

Next, we define the concept of a 
conservative solution.

\begin{definition}[Global conservative weak solution 
of \eqref{t-c-N-n}--\eqref{iid}]\label{defsc}
Suppose that $(u_0,v_0)\in\Sigma$, where $\Sigma$ 
is defined in \eqref{Sig}.
We define a vector function $(u,v)(t,x)$ as a 
global conservative weak solution to the Cauchy 
problem \eqref{t-c-N-n}-\eqref{iid} if it meets 
two criteria. First, $(u,v)$ must be a global weak solution 
in accordance with Definition \ref{defs}, and secondly, it 
must adhere to the following five conditions:
\begin{enumerate}
	\item $\int_{-\infty}^{\infty}
	\left(uv+(\px u)\px v\right)(t,x)\,dx=G_{0}$, 
	for any $t\in\mathbb{R}$;
	
	\item there exist positive Radon measures
	$\lambda_t^{(u)}$,
	$\lambda_t^{(v)}$ and $\lambda_{t}^{(uv)}$
	on $\R$
	such that
	\begin{enumerate}
		\item the absolutely continuous parts of
		$\lambda_t^{(u)}$, $\lambda_t^{(v)}$ and $\lambda_t^{(uv)}$ with respect to the 
		Lebesgue measure on $\R$ have the following form:
		\begin{equation*}
		\begin{split}
			d\lambda_t^{(u,ac)}=(\px u)^2\,dx,
			\quad
			d\lambda_t^{(v,ac)}=(\px v)^2\,dx, \quad 
			d\lambda_t^{(uv,ac)}
			=\left((\px u)^2(\px v)^2\right)dx,
		\end{split}
		\end{equation*}
		while the nonzero singular parts of
		$\lambda_t^{(u)}$, $\lambda_t^{(v)}$ 
		and $\lambda_t^{(uv)}$ are supported, 
		for a.e.\,\,$t\in\R$, on the sets where
		$v(t,\cdot)=0$, $u(t,\cdot)=0$ and 
		$(uv)(t,\cdot)=0$ respectively;
	
		\item the following conservation 
		laws hold for any $t\in\mathbb{R}$:
		\begin{equation*}
			\begin{split}
			&\int_{-\infty}^{\infty}u^2(t,x)\,dx
			+\lambda_t^{(u)}(\R)
			=E_{u_0},\quad
			\int_{-\infty}^{\infty}v^2(t,x)\,dx
			+\lambda_t^{(v)}(\R)
			=E_{v_0},\\
			&\int_{-\infty}^{\infty}
			\left(3u^2v^2
			+4uv(\px u)\px v\right)(t,x)\,dx
			+\int_{-\infty}^{\infty}u^2(t,x)
			\,d\lambda_t^{(v)}
			+\int_{-\infty}^{\infty}v^2(t,x)
			\,d\lambda_t^{(u)}
			-\lambda_t^{(uv)}(\R)
			=H_0;
			\end{split}
		\end{equation*}
		\item the following inequality
		holds for any $t\in\R$:
		\begin{equation}
			\label{csho}
			\int_{-\infty}^{\infty}
			\left(3u^2v^2
			+4uv(\px u)\px v-(\px u)^2(\px v)^2\right)
			(t,x)\,dx
			+\int_{-\infty}^{\infty}u^2(t,x)
			\,d\lambda_t^{(v)}
			+\int_{-\infty}^{\infty}v^2(t,x)
			\,d\lambda_t^{(u)}
			\geq H_0;
		\end{equation}
	\end{enumerate}
	
	\item 
	the following inequalities 
	hold for any $t\in\R$:
	\begin{align}\label{consin}
		&\|u(t,\cdot)\|_{H^1}^2\leq E_{u_0},\quad
		\|v(t,\cdot)\|_{H^1}^2\leq E_{v_0}.
	\end{align}
\end{enumerate}
Moreover, introducing the sets
\begin{equation}\label{DN}
	\begin{split}
		&D_W(t)=\left\{
		\xi\in\mathbb{R}:\cos\frac{W(t,\xi)}{2}=0\right\},
		\quad
		D_Z(t)=\left\{
		\xi\in\mathbb{R}:\cos\frac{Z(t,\xi)}{2}=0\right\},
		\\
		&N_W=\left\{
		t\in\mathbb{R}:\mathrm{meas}(D_W(t))> 0\right\},
		\quad
		N_Z=\left\{
		t\in\mathbb{R}:\mathrm{meas}(D_Z(t))> 0\right\},
	\end{split}
\end{equation}
where $(U,V,W,Z,q)$, as defined in Theorem \ref{gwp} below, 
depends solely on the initial conditions $(u_0,v_0)$, 
and $\mathrm{meas}(\cdot)$ denotes 
the Lebesgue measure on $\mathbb{R}$, we also have
\begin{enumerate}\setcounter{enumi}{3}
	\item $\|u(t,\cdot)\|_{H^1}^2=E_{u_0}$, 
	for any $t\in\mathbb{R}\setminus N_W$;
	\quad $\|v(t,\cdot)\|_{H^1}^2=E_{v_0}$,
	for any $t\in\mathbb{R}\setminus N_Z$;
	$$
	\int_{-\infty}^{\infty}
	\left(3u^2v^2+u^2(\px v)^2+v^2(\px u)^2
	+4uv(\px u)\px v-(\px u)^2(\px v)^2\right)
	(t,x)\,dx=H_0,
	$$
	\qquad\qquad\qquad\qquad\qquad\qquad
	\qquad\qquad\qquad\qquad\qquad\qquad 
	for any $t\in\mathbb{R}\setminus (N_W\cup N_Z)$;
	\item 
	\begin{itemize}
		\item if $\mathrm{meas}(N_W)=0$, 
		then $\lambda_t^{(u)}$ is a measure-valued 
		solution $w_u$ of the following equation:
		$$
		\pt w_u+\px(uvw_u)
		=2(\px u)(u^2v-P_1-\px P_2)+u(\px u)^2\px v;
		$$
	
		\item if $\mathrm{meas}(N_Z)=0$, then
		$\lambda_t^{(v)}$ is a measure-valued solution 
		$w_v$ of the following equation:
		$$
		\pt w_v+\px(uvw_v)
		=2(\px v)(uv^2-S_1-\px S_2)+v\px u(\px v)^2;
		$$

		\item if $\mathrm{meas}(N_W)$, 
		$\mathrm{meas}(N_Z)=0$, then
		$\lambda_t^{(uv)}$ is a measure-valued 
		solution $w_{uv}$ of the following equation:
		\begin{equation}\label{wuv}
			\pt w_{uv}+\px(uvw_{uv})
			=2(\px u)(\px v)
			\left((u^2v-P_1-\px P_2)\px v
			+(uv^2-S_1-\px S_2)\px u
			\right);
		\end{equation}
		
\end{itemize}
\end{enumerate}
\end{definition}

\begin{remark}[Sets $N_W$ and $N_Z$]\label{NWNZ}
A solution $(u,v)(t,x)$ can experience wave breaking only 
for $t\in N_W\cup N_Z$. More precisely, for $t\in N_W$, 
we have the concentration of $(\partial_x u)^2\,dx$ 
and the product $(\partial_x u)^2(\partial_x v)^2\,dx$, while 
for $t\in N_Z$, we have the concentration of 
$(\partial_x v)^2\,dx$ and $(\partial_x u)^2(\partial_x v)^2\,dx$, 
see (2) and (4) of Definition \ref{defsc}.
\end{remark}

\begin{remark}[Measure-valued solution]\label{wms}
In items (5) of Definition \ref{defsc} and 
Theorem \ref{Thm} below, we consider $w$ to be a 
measure-valued solution of a 
linear transport equation with a source, given by
$$
\partial_t w + \partial_x (uwv) = f,
$$
when the following identity is satisfied:
\begin{equation*}
	\int_{-T}^T\int_{-\infty}^{\infty}
	(\partial_t \phi + uv \partial_x \phi)
	\, dw \, dt + \int_{-T}^T\int_{-\infty}^{\infty} 
	f \phi \, dx \, dt = 0,
\end{equation*}
for any $\phi \in C^\infty((-T,T)\times\mathbb{R})$ 
with compact support and for any arbitrary $T > 0$.
\end{remark}

Now, we are in a position to present 
the primary result of this paper.
Let us consider a positive Radon measure $\mu$ 
on $\mathbb{R}$ such that (cf.~\cite[Section 6]{BC07})
\begin{equation}\label{mu}
	\mu=\mu^{ac}+\mu^s,\quad
	d\mu^{ac}=
	\left(
	(\px u)^2+(\px v)^2+(\px u)^2(\px v)^2
	\right)dx,
\end{equation}
where $\mu^{ac}$ and $\mu^s$ respectively denote 
the absolutely continuous and singular parts 
of $\mu$ with respect to the Lebesgue measure 
on $\mathbb{R}$ \cite[Section 1.6, Theorem 3]{EG92}.
Introduce the set
\begin{equation}\label{D}
	\mathcal{D}
	=\left\{(u,v,\mu):(u,v)\in\Sigma
	\mbox{ and }\mu
	\mbox{ is a positive Radon measure which satisfies } 
	\eqref{mu}
	\right\},
\end{equation}
where the metric space $(\Sigma,d_\Sigma)$ is 
defined by \eqref{Sig} and \eqref{d^2}.

This is our main theorem:

\begin{theorem}[Global semigroup of conservative solutions]
\label{Thm}
Consider $(u_0,v_0,\mu_0)\in\mathcal{D}$. 
Then there exists a flow map 
$\Psi_t:\mathbb{R}\times\mathcal{D}
\to\mathcal{D}$ whose trajectories
$\Psi_t(u_0,v_0,\mu_0)=(u(t),v(t),\mu_{(t)})$ 
satisfy the following properties:
\begin{enumerate}
	\item $(u,v)$ is a global conservative weak solution 
	of \eqref{t-c-N-n}--\eqref{iid} in the 
	sense of Definition \ref{defsc};
	
	\item $\Psi_t$ satisfies the semigroup property, that 
	is (i) $\Psi_0=\mathrm{id}$ and 
	(ii) $\Psi_{t+\tau}=\Psi_{t}\circ\Psi_{\tau}$;

	\item assuming that 
	\begin{enumerate}[a)]
		\item $d_\Sigma\left((u_{0,n},v_{0,n}),
		(u_0,v_0)\right)\to0$ as $n\to\infty$,
		
		\item $\mu_{0,n}
		\overset{\ast}{\rightharpoonup} \mu_0$ 
		(weakly-$\ast$) as $n\to\infty$,
	\end{enumerate}
	we have $\forall T>0$,
	\begin{equation}\label{unifc1}
		\|(u_n-u)\|
		_{L^\infty([-T,T]\times\mathbb{R})}
		+\|(v_n-v)\|
		_{L^\infty([-T,T]\times\mathbb{R})}\to 0,
		\quad n\to\infty.
	\end{equation}
	Here, $(u_n,v_n)$ are the solutions of 
	\eqref{t-c-N-n} that correspond to the 
	initial data $(u_{0,n},v_{0,n})$;

	\item $\mu_{(t)}(\mathbb{R})\leq 
	E_{u_0}+E_{v_0}+7E_{u_0}E_{v_0}-H_0$ for any 
	$t\in\mathbb{R}$;

	\item if $\mathrm{meas}(N_W)$, $\mathrm{meas}(N_Z)=0$, 
	then $\mu_{(t)}$ is a measure-valued solution $w$ 
	of the following transport equation with source term:
	\begin{equation}\label{mvm1}
		\begin{split}
			\pt w+\px(uvw)
			=&2\px u(1+(\px v)^2)(u^2v-P_1-\px P_2)
			+2\px v(1+(\px u)^2)(uv^2-S_1-\px S_2)
			\\&+u(\px u)^2\px v+v\px u(\px v)^2.
		\end{split}
	\end{equation}
\end{enumerate}
\end{theorem}

\begin{remark}
Notice that unlike the problem for the Camassa-Holm 
equation \cite{BC07}, the sets $N_W$ and/or $N_Z$ can 
have positive measure. See also Remark \ref{NWNZ} above.
The same situation occurs in the case of the 
Novikov equation, where $u=v$ 
and $N_W=N_Z$ \cite{CCL18}.
\end{remark}

\begin{remark}
Note that for the Novikov equation \eqref{N}, item (1) 
in Definition \ref{defsc} ensures the 
preservation of the $H^1$ norm of $u$ at all times $t$, 
which aligns with the findings in \cite{CCL18}. 
Consequently, the concentration of 
$(\partial_x u)^4$ is exclusive to the 
solution of the Novikov equation. In contrast, for the 
two-component system \eqref{t-c-N-n}--\eqref{iid}, 
concentrations can occur 
for $(\partial_x u)^2$, $(\partial_x v)^2$, 
and $(\partial_x u)^2(\partial_x v)^2$. 
See also Remark \ref{NWNZ}.
\end{remark}

\begin{remark}
For the Novikov equation \eqref{N}, the Radon measure 
$\lambda_{t}^{(uu)}$ satisfies \eqref{wuv} for 
any values of $\mathrm{meas}(N_W)$ \cite{CCL18, CCCS18}. 
However, when $u$ and $v$ are unrelated, we must impose 
the additional assumption that $\mathrm{meas}(N_W)=0$ 
and $\mathrm{meas}(N_Z)=0$ to ensure that certain integrals in 
the weak formulation of the solution vanish. 
This is demonstrated in the proof of item 
(4) in 
Proposition \ref{mcl} below (see \cite{HQ21}).
For the same reasons, we need to impose 
assumptions on $\mathrm{meas}(N_W)$ and/or 
$\mathrm{meas}(N_Z)$ to assert that the 
Radon measures $\lambda_{t}^{(u)}$, 
$\lambda_{t}^{(v)}$, and $\mu_{(t)}$ 
satisfy the corresponding transport 
equations with a source (see item (5) 
of Definition \ref{defsc}).
\end{remark}

\begin{remark}
By referring to \eqref{u_x^2} below 
and utilizing \eqref{consin} 
through \eqref{csho}, we can conclude that
$$
\|((\px u)\px v)(t,\cdot)\|_{L^2}^2
\leq 7E_{u_0}E_{v_0}-H_0,\quad
\mbox{for any }t\in\mathbb{R}.
$$
\end{remark}

\section{Change of variables and a semilinear system}\label{chv}
\subsection{Conservation laws}
\label{cls}

We begin by presenting the fundamental conservation laws 
of \eqref{t-c-N} and some a priori estimates. 
Differentiating \eqref{t-c-N-n} with respect to $x$, 
we obtain
\begin{equation}\label{tcNnd}
	\begin{split}
		&\pt\px u+uv\px^2 u-u^2v+\frac{1}{2}v(\px u)^2
		+P_1+\px P_2=0,\\
		&\pt\px v+uv\px^2v-uv^2+\frac{1}{2}u(\px v)^2
		+S_1+\px S_2=0,
	\end{split}
\end{equation}
where $P_j$ and $S_j$, $j=1,2$, are 
given in \eqref{P12} and \eqref{S12} respectively. 
These equations, along with \eqref{t-c-N}, can be 
used to prove the following proposition:

\begin{proposition}[\cite{HQ21,HLQ23,L19}]\label{PCL}
Assume that $(u_0,v_0)\in \Sigma$, 
where the metric space $(\Sigma,d_\Sigma)$ 
is defined by \eqref{Sig} and \eqref{d^2}.
Then the solution of the two-component Novikov equation 
formally satisfies the following local conservation laws:
\begin{equation*}
	\begin{split}
		&\int_{-\infty}^{\infty}
		\left(u^2+(\px u)^2\right)(t,x)\,dx=E_{u_0},\quad
		\int_{-\infty}^{\infty}
		\left(v^2+(\px v)^2\right)(t,x)\,dx=E_{v_0},
	\end{split}
\end{equation*}
and
\begin{equation*}
	\int_{-\infty}^{\infty}
	\left(uv+(\px u)\px v\right)(t,x)\,dx=G_{0},
\end{equation*}
as well as
\begin{equation}\label{hc}
	\int_{-\infty}^{\infty}
	\left(3u^2v^2+u^2(\px v)^2+v^2(\px u)^2
	+4uv(\px u)\px v-(\px u)^2(\px v)^2\right)
	(t,x)\,dx=H_0,
\end{equation}
for all $t$, where $E_{u_0}$, $E_{v_0}$, $G_0$ 
and $H_0$ are given in \eqref{consq}.
\end{proposition}

We conclude this subsection by providing some uniform 
bounds for $P_j$, $\px P_j$ and $S_j$, $\px S_j$, 
$j=1,2$. By employing 
the Sobolev embedding theorem and the 
Cauchy-Schwarz inequality, we obtain the following:
\begin{equation}\label{L1-est}
	\begin{split}
		\left\|u^2v+u(\px u)\px v+\frac{1}{2}v(\px u)^2
		\right\|_{L^1}&
		\leq\frac{1}{2}
		\left\|\left(u^2+(\px u)^2\right)v\right\|_{L^1}
		+\frac{1}{2}\left\|u^2v\right\|_{L^1}
		+\|u(\px u)\px v\|_{L^1}\\
		&\leq
		\frac{1}{2}\|v\|_{L^\infty}
		\left(\|u\|^2_{H^1}
		+\|u\|^2_{L^2}\right)
		+\|u\|_{L^\infty}\|\px u\|_{L^2}\|\px v\|_{L^2}\\
		&\leq 2E_{u_0}E_{v_0}^{1/2}.
	\end{split}
\end{equation}
Combining \eqref{P12} and \eqref{L1-est} yields
\begin{equation*}
	\begin{split}
		\|P_1(t,\cdot)\|_{L^p},\,\,
		\|\px P_1(t,\cdot)\|_{L^p}\leq
		\frac{1}{2}\left\|e^{-|\cdot|}\right\|_{L^p}
		\left\|u^2v+u(\px u)\px v+\frac{1}{2}v(\px u)^2
		\right\|_{L^1}
		\leq C_pE_{u_0}E_{v_0}^{1/2},
	\end{split}
\end{equation*}
for some constant $C_p>0$, $p\in [1,\infty]$.
Arguing similarly for $S_1$, we obtain
\begin{equation*}
	\begin{split}
		\|S_1(t,\cdot)\|_{L^p},\,\,
		\|\px S_1(t,\cdot)\|_{L^p}\leq
		\frac{1}{2}\left\|e^{-|\cdot|}\right\|_{L^p}
		\left\|uv^2+v(\px u) \px v+\frac{1}{2}u(\px v)^2
		\right\|_{L^1}
		\leq C_pE_{u_0}^{1/2}E_{v_0},
	\end{split}
\end{equation*}
where we have used \eqref{L1-est}.

To obtain similar a priori 
estimates for $P_2$ and $S_2$, which involve 
$(\px u)^2\px v$ and $(\px u)(\px v)^2$ respectively, one 
should employ the higher-order conservation law \eqref{hc}, 
as discussed in \cite{CCL18}. Indeed, using
\begin{equation}\label{u_x^2}
	\begin{split}
		\int_{-\infty}^{\infty}(\px u)^2(\px v)^2\,dx&=
		\int_{-\infty}^{\infty}\left(3u^2v^2+u^2(\px v)^2
		+v^2(\px u)^2+4uv(\px u)\px v\right)\,dx-H_0\\
		&\leq\|u\|^2_{L^{\infty}}\|v\|^2_{H^1}
		+\int_{-\infty}^{\infty}\left(2u^2v^2
		+v^2(\px u)^2+4uv(\px u)\px v\right)\,dx-H_0
		\\ &\leq\|u\|^2_{L^{\infty}}\|v\|^2_{H^1}
		+\|u\|^2_{H^1}\|v\|^2_{L^{\infty}}
		+\|u\|^2_{L^{\infty}}\|v\|^2_{L^2}
		\\ &\quad+4\|uv\|_{L^\infty}
		\|\px u\|_{L^2}\|\px v\|_{L^2}
		-H_0\leq 7E_{u_0}E_{v_0}-H_0,
	\end{split}
\end{equation}
we have by the Cauchy-Schwarz inequality that
\begin{equation*}
	\|(\px u)^2\px v\|_{L^1}\leq\|\px u\|_{L^2}
	\|(\px u)\px v\|_{L^2}\leq
	\left(E_{u_0}(7E_{u_0}E_{v_0}-H_0)\right)^{1/2},
\end{equation*}
and
\begin{equation*}
	\|(\px u)(\px v)^2\|_{L^1}\leq\|\px v\|_{L^2}
	\|(\px u)\px v\|_{L^2}\leq
	\left(E_{v_0}(7E_{u_0}E_{v_0}-H_0)\right)^{1/2}.
\end{equation*}
This yields, for any 
$p\in[1,\infty]$ and some $C_p>0$, that
\begin{equation*}
	\begin{split}
		\|P_2(t,\cdot)\|_{L^p},\,\,
	\|\px P_2(t,\cdot)\|_{L^p}\leq
	\frac{1}{4}\left\|e^{-|\cdot|}\right\|_{L^p}
	\left\|(\px u)^2\px v\right\|_{L^1}
	\leq C_pK_{u_0},
	\end{split}
\end{equation*}
and
\begin{equation*}
	\begin{split}
	\|S_2(t,\cdot)\|_{L^p},\,\,
	\|\px S_2(t,\cdot)\|_{L^p}\leq
	\frac{1}{4}\left\|e^{-|\cdot|}\right\|_{L^p}
	\left\|(\px u)(\px v)^2\right\|_{L^1}
	\leq C_pK_{v_0},
	\end{split}
\end{equation*}
where $K_{u_0}$ and $K_{v_0}$ 
are given in \eqref{Ku}.

\subsection{New variables}
In this subsection, we revisit the formalism 
regarding the change of variables for \eqref{t-c-N}, 
alongside its reduction to an equivalent ODE system, 
as outlined in \cite[Section III]{HQ21}. The arguments presented 
herein are formal and will be substantiated 
in Section \ref{GsN}.

Define the characteristic $y(t, \xi)$ 
as the solution to the following Cauchy problem:
\begin{equation}\label{char}
	\begin{split}
		\pt y(t,\xi)=u(t,y(t,\xi))v(t,y(t,\xi)),
		\quad y(0,\xi)=y_0(\xi),
\end{split}
\end{equation}
where the initial data $y_0(\xi)$ is a strictly 
monotone increasing function given 
by (cf.~\cite[Equation (3.1)]{BC07}):
\begin{equation}\label{y0}
	\int_0^{y_0(\xi)}\left(1+(\px u_0)^2(x)\right)
	\left(1+(\px v_0)^2(x)\right)\,dx=\xi.
\end{equation}
Then we introduce the following new 
unknowns (cf.~\cite{BC07,CCL18, HQ21}):
\begin{equation}\label{UV}
	U(t,\xi)=u(t,y(t,\xi)),\quad
	V(t,\xi)=v(t,y(t,\xi)),
\end{equation}
\begin{equation}\label{WZ}
	W(t,\xi)=2\arctan \px u(t,y(t,\xi)),\quad
	Z(t,\xi)=2\arctan \px v(t,y(t,\xi)),
\end{equation}
and
\begin{equation}\label{q}
	q(t,\xi)=\left(1+(\px u)^2\right)
	\left(1+(\px v)^2)(t,y(t,\xi)\right)\pxi y(t,\xi).
\end{equation}

In the subsequent discussion, for the sake of simplicity, we 
will typically omit the arguments of $U, V, W, Z, q$, 
and $y$. Equations \eqref{UV} and \eqref{WZ} lead to 
the following (where $\px u = (\px u)(t,y)$, 
$\px v = (\px v)(t,y)$, and $y$ 
is determined by \eqref{char}):
\begin{equation}\label{trig}
	\begin{split}
		&\cos^2\frac{W}{2}=\frac{1}{1+(\px u)^2},\quad
		\sin^2\frac{W}{2}=\frac{(\px u)^2}{1+(\px u)^2},\quad
		\sin W=\frac{2\px u}{1+(\px u)^2},\\
		&\cos^2\frac{Z}{2}=\frac{1}{1+(\px v)^2},\quad
		\sin^2\frac{Z}{2}=\frac{(\px v)^2}{1+(\px v)^2},\quad
		\sin Z=\frac{2\px v}{1+(\px v)^2}.
	\end{split}
\end{equation}
Combining \eqref{q} and \eqref{trig} we conclude 
that (see \cite[Equation (3.2)]{HQ21})
\begin{equation}\label{pxiy}
	\pxi y=q\cos^2\frac{W}{2}\cos^2\frac{Z}{2},
\end{equation}
which implies
\begin{equation}\label{ydiff}
	y(t,\xi)-y(t,\eta)=\int_\eta^\xi
	\left(q\cos^2\frac{W}{2}
	\cos^2\frac{Z}{2}\right)(t,s)\,ds,
	\quad \xi,\eta\in\mathbb{R}.
\end{equation}

\subsection{A semilinear system}
From \eqref{t-c-N-n}, \eqref{char} and \eqref{UV}, 
we can deduce that
\begin{equation}\label{UVd}
	\pt U=-(\px P_1+P_2)(t,y),\quad
	\pt V=-(\px S_1+S_2)(t,y).
\end{equation}
Taking into account \eqref{trig}
and using \eqref{tcNnd}, \eqref{char}, \eqref{WZ}, 
we arrive at
\begin{equation}\label{Wd}
	\begin{split}
		\pt W&=\frac{2}{1+(\px u)^2}
		\left(\pt\px u+uv\px^2 u\right)(t,y)\\
		&=\frac{2}{1+(\px u)^2}\left(u^2v
		-\frac{1}{2}v(\px u)^2
		-P_1-\px P_2\right)(t,y)\\
		&=2U^2V\cos^2\frac{W}{2}-V\sin^2\frac{W}{2}
		-2\cos^2\frac{W}{2}
		\left(P_1+\px P_2\right)(t,y),
	\end{split}
\end{equation}
as well as, after similar calculations,
\begin{equation}\label{Zd}
	\pt Z=2UV^2\cos^2\frac{Z}{2}-U\sin^2\frac{Z}{2}
	-2\cos^2\frac{Z}{2}
	\left(S_1+\px S_2\right)(t,y).
\end{equation}
Equations \eqref{tcNnd}, \eqref{q} 
and \eqref{trig} imply 
\begin{equation*}
	\begin{split}
		\pt q&=2(\px u)\left(\pt\px u+uv\px^2 u\right)
		\left(1+(\px v)^2\right)(t,y)\pxi y\\
		&\quad+
		2(\px v)\left(\pt\px v+uv\px^2 v\right)
		\left(1+(\px u)^2\right)(t,y)\pxi y\\
		&=q\left(U^2V-\frac{1}{2}v(\px u)^2
		-(P_1+\px P_2)(t,y)\right)\sin W\\
		&\quad+q\left(UV^2-\frac{1}{2}u(\px v)^2
		-(S_1+\px S_2)(t,y)\right)\sin Z
		+q\left(v\px u+u\px v\right)(t,y).
	\end{split}
\end{equation*}
Then using
\begin{equation*}
	\begin{split}
		&q(\px u)v\left(1-\frac{1}{2}(\px u)\sin W\right)
		=\frac{qv\px u}{1+(\px u)}
		=\frac{1}{2}qV\sin W,
		\\ & qu(\px v)\left(1-\frac{1}{2}(\px v)\sin Z\right)
		=\frac{qu\px v}{1+(\px v)}
		=\frac{1}{2}qU\sin Z,
	\end{split}
\end{equation*}
we obtain
\begin{equation}\label{qd}
	\begin{split}
		\pt q=&q\left(U^2V+\frac{1}{2}V
		-(P_1+\px P_2)(t,y)\right)\sin W
		\\ &
		+q \left(UV^2+\frac{1}{2}U
		-(S_1+\px S_2)(t,y)\right)\sin Z.
	\end{split}
\end{equation}

To derive a system of ODEs in a Banach space 
from \eqref{UVd}, \eqref{Wd}, \eqref{Zd} and \eqref{qd}, 
we must obtain expressions for $P_j(t,y)$, $S_j(t,y)$, 
$\px P_j(t,y)$, and $\px S_j(t,y)$, $j=1,2$, 
in terms of $(U,V,W,Z,q)$.

\begin{proposition}\label{Pj}
Suppose that $y(t, \cdot)$ is a strictly monotone 
increasing function that is 
a bijection from $\mathbb{R}$ to itself.
Introduce the exponential function
\begin{equation}\label{E}
	\mathcal{E}(t,\xi,\eta)
	=\exp\left(-\left|\int_\eta^\xi
	\left(q\cos^2\frac{W}{2}
	\cos^2\frac{Z}{2}\right)(t,s)\,ds\right|\right),
	\quad t,\xi,\eta\in\mathbb{R},
\end{equation}
as well as
\begin{equation}\label{p1q1}
	\begin{split}
		&p_1(t,\xi)=q(t,\xi)\left(
		U^2V
		\cos^2\frac{W}{2}
		\cos^2\frac{Z}{2}
		+\frac{1}{4}U\sin W\sin Z
		+\frac{1}{2}V\sin^2\frac{W}{2}
		\cos^2\frac{Z}{2}
		\right)(t,\xi),
		\\ & 
		s_1(t,\xi)=q(t,\xi)
		\left(UV^2
		\cos^2\frac{W}{2}
		\cos^2\frac{Z}{2}
		+\frac{1}{2}U\cos^2\frac{W}{2}
		\sin^2\frac{Z}{2}
		+\frac{1}{4}V\sin W\sin Z
		\right)(t,\xi),
	\end{split}
\end{equation}
and
\begin{equation}\label{p2q2}
	\begin{split}
		& p_2(t,\xi)=\left(q\sin^2\frac{W}{2}
		\sin Z\right)(t,\xi),\\
		& s_2(t,\xi)=\left(q\sin W\sin^2\frac{Z}{2}
		\right)(t,\xi).
	\end{split}
\end{equation}
Then, the nonlocal terms $P_j$, $\px P_j$, $S_j$, 
and $\px S_j$, $j=1,2$, can be expressed in the following 
form, as per \cite[Equation (3.3)]{HQ21} (we slightly deviate 
from the notation by writing, for instance, 
$P_1(t,\xi)$ instead of $P_1(t,y(t,\xi))$):
\begin{align}\label{P_1}
	&P_1(t,\xi)=\frac{1}{2}
	\int_{-\infty}^{\infty}
	\mathcal{E}(t,\xi,\eta)
	p_1(t,\eta)\,d\eta,\\
	\label{pxP_1}
	&(\px P_1)(t,\xi)=\frac{1}{2}
	\left(
	\int_{\xi}^{\infty}
	-\int_{-\infty}^{\xi}
	\right)
	\mathcal{E}(t,\xi,\eta)
	p_1(t,\eta)\,d\eta,
\end{align}
\begin{align}
	\nonumber
	&P_2(t,\xi)=\frac{1}{8}
	\int_{-\infty}^{\infty}
	\mathcal{E}(t,\xi,\eta)
	p_2(t,\eta)\,d\eta,\\
	\label{pxP_2}
	&(\px P_2)(t,\xi)=\frac{1}{8}
	\left(
	\int_{\xi}^{\infty}
	-\int_{-\infty}^{\xi}
	\right)
	\mathcal{E}(t,\xi,\eta)
	p_2(t,\eta)\,d\eta,
\end{align}
and
\begin{align*}
	&S_1(t,\xi)=\frac{1}{2}\int_{-\infty}^{\infty}
	\mathcal{E}(t,\xi,\eta)
	s_1(t,\eta)\,d\eta,\\
	&(\px S_1)(t,\xi)=\frac{1}{2}
	\left(
	\int_{\xi}^{\infty}
	-\int_{-\infty}^{\xi}
	\right)
	\mathcal{E}(t,\xi,\eta)
	s_1(t,\eta)\,d\eta
\end{align*}
\begin{align*}
	&S_2(t,\xi)=\frac{1}{8}
	\int_{-\infty}^{\infty}
	\mathcal{E}(t,\xi,\eta)
	s_2(t,\eta)\,d\eta,\\
	&(\px S_2)(t,\xi)=\frac{1}{8}
	\left(
	\int_{\xi}^{\infty}
	-\int_{-\infty}^{\xi}
	\right)
	\mathcal{E}(t,\xi,\eta)
	s_2(t,\eta)\,d\eta.
\end{align*}
\end{proposition}

\begin{proof}
The nonlocal term $P_1(t,y)$ 
can be written as follows:
\begin{equation}\label{P_1-1}
	\begin{split}
		P_1(t,y)&=\frac{1}{2}\int_{-\infty}^{\infty}
		e^{-|y(t,\xi)-z|}
		\left(
		u^2v+u(\px u)\px v+\frac{1}{2}v(\px u)^2
		\right)(t,z)\,dz\\
		&=\frac{1}{2}\int_{-\infty}^{\infty}
		e^{-|y(t,\xi)-y(t,\eta)|}
		\left(q\left(
		U^2V\cos^2\frac{W}{2}
		\cos^2\frac{Z}{2}\right.\right.\\
		&\left.\left.\qquad\qquad\qquad\qquad\qquad\qquad\,\,
		+\frac{1}{4}U\sin W\sin Z
		+\frac{1}{2}V\sin^2\frac{W}{2}
		\cos^2\frac{Z}{2}
		\right)\right)(t,\eta)
		\,d\eta,
	\end{split}
\end{equation}
where we have executed the change of variables 
$z = y(t,\eta)$, which is justified by the assumption 
on $y(t,\cdot)$. Additionally, we have applied the 
equations \eqref{trig} and \eqref{pxiy}.
Substituting the right-hand side of \eqref{ydiff} into 
the exponent of \eqref{P_1-1}, and using the 
notations from \eqref{E} and \eqref{p1q1}, 
we derive \eqref{P_1}.

Observing that
\begin{equation*}
	(\px P_1)(t,y)
	=\frac{1}{2}
	\left(
	\int_{y(t,\xi)}^{\infty}
	-\int_{-\infty}^{y(t,\xi)}
	\right)
	e^{-|y(t,\xi)-z|}
	\left(
	u^2v+u(\px u)\px v+\frac{1}{2}v(\px u)^2
	\right)(t,z)\,dz,
\end{equation*}
we can reason along the lines of the calculations 
for $P_1$ to derive \eqref{pxP_1}.

Arguing in a similar manner, we can derive the 
expressions for $P_2, \px P_2$ and $S_j,\px S_j$, $j=1,2$, 
as specified in the proposition.
\end{proof}

Combining \eqref{UVd}, \eqref{Wd}, \eqref{Zd}, and \eqref{qd}, 
we obtain the following ODE system in a Banach space 
(see \cite[Equations (4.1)—(4.2)]{HQ21}):
\begin{equation}\label{ODE}
	\begin{split}
		&\pt U=-\px P_1-P_2,\\
		&\pt V=-\px S_1-S_2,\\
		&\pt W=
		2U^2V\cos^2\frac{W}{2}-V\sin^2\frac{W}{2}
		-2\left(P_1+\px P_2\right)\cos^2\frac{W}{2},\\
		&\pt Z=
		2UV^2\cos^2\frac{Z}{2}-U\sin^2\frac{Z}{2}
		-2\left(S_1+\px S_2\right)\cos^2\frac{Z}{2},\\
		&\pt q=
		q\left(U^2V+\frac{1}{2}V-P_1-\px P_2\right)
		\sin W
		+q\left(UV^2+\frac{1}{2}U
		-S_1-\px S_2\right)\sin Z,
	\end{split}
\end{equation}
subject to the initial data (see \eqref{y0}, 
\eqref{UV}, \eqref{WZ} and \eqref{q})
\begin{equation}\label{id}
	\begin{split}
		&U_0(\xi):= U(0,\xi)=u_0(y_0(\xi)),\quad
		V_0(\xi):= V(0,\xi)=v_0(y_0(\xi)),\\
		&W_0(\xi):= W(0,\xi)
		=2\arctan(\px u_0)(y_0(\xi)),\quad
		Z_0(\xi):= Z(0,\xi)
		=2\arctan(\px v_0)(y_0(\xi)),\\
		& q_0(\xi):= q(0,\xi)=1,
	\end{split}
\end{equation}
where $P_j$, $\px P_j$, $S_j$, $\px S_j$, $j=1,2$, are 
given in Proposition \ref{Pj}, and we 
assume that $W(0,\xi), Z(0,\xi)\in[-\pi,\pi]$.

\begin{remark}
We assume the hypothesis of Proposition \ref{Pj}, which 
guarantees the validity of substituting the variable $y(t,\cdot)$ 
with $z$. This assumption will be confirmed later, as 
detailed in Proposition \ref{propy}.
\end{remark}

\begin{remark}\label{Rpi}
The vector field defined on the right-hand side of \eqref{ODE} 
is invariant under adding multiples of $2\pi$ to 
either $W$ or $Z$. Consequently, we can permit $W$ and $Z$ 
to assume values beyond the interval $[-\pi,\pi]$ 
by employing periodic extension.
\end{remark}

\section{Local existence and uniqueness of ODE system}\label{Gsl}

In \cite{HQ21}, the authors consider 
initial data $u_0,v_0$ from the space $H^1\cap W^{1,4}$. 
This choice immediately facilitates the consideration 
of \eqref{ODE}--\eqref{id} as an ODE system in 
the following Banach space (see the space $Y$ 
described in \cite[Section IV]{HQ21}:
\begin{equation}
	\label{EB}
	E=\left(H^{1}(\mathbb{R})
	\cap W^{1,4}(\mathbb{R})\right)^2\times
	\left(L^2(\mathbb{R})
	\cap L^{\infty}(\mathbb{R})\right)^2\times
	L^{\infty}(\mathbb{R}),
\end{equation}
equipped with the norm
\begin{equation*}
	\|(U,V,W,Z,q)\|_{E}=\|U\|_{H^1\cap W^{1,4}}
	+\|V\|_{H^1\cap W^{1,4}}
	+\|W\|_{L^2\cap L^{\infty}}
	+\|Z\|_{L^2\cap L^{\infty}}
	+\|q\|_{L^{\infty}}.
\end{equation*}

For initial data $(u_0, v_0)$ in $\Sigma$, as defined in \eqref{Sig}, 
which is the subject of this paper, neither $\partial_x u_0$ 
nor $\partial_x v_0$ is generally in $L^4$.
Nevertheless, taking into 
account that (see \eqref{y0})
$$
\pxi y_0(\xi)=\frac{1}
{\left(1+(\px u_0)^2\right)
\left(1+(\px v_0)^2\right)(y_0(\xi))},
$$
we can bound the $L^4$ norm of $\pxi U_0$ by the $L^2$ 
norm of $\px u_0$ as follows (recall that 
$U_0(\xi)=u_0(y_0(\xi))$):
\begin{equation}\label{U-L4}
	\begin{split}
	\int_{-\infty}^{\infty}
	(\pxi U_0)^4\,d\xi&=
	\int_{-\infty}^{\infty}
	\left(\px u_0(y_0)\right)^4
	(\pxi y_0)^4\,d\xi
	\leq
	\int_{-\infty}^{\infty}
	\left(\px u_0(y_0)\right)^4
	(\pxi y_0)^2\,d\xi\\
	&=\int_{-\infty}^{\infty}
	\frac{(\px u_0)^4(x)}
	{\left(1+(\px u_0)^2\right)
	\left(1+(\px v_0)^2\right)(x)}\,dx\\
	&\leq
	\int_{-\infty}^{\infty}
	\frac{(\px u_0)^2(x)}
	{\left(1+(\px v_0)^2\right)(x)}\,dx
	\leq
	\int_{-\infty}^{\infty}
	(\px u_0)^2(x)\,dx.
	\end{split}
\end{equation}
Here, we have utilized the fact that $|\pxi y_0|\leq 1$ 
in the first inequality. Similarly, we can argue 
for $\pxi V_0$ to conclude that $U_0, V_0 \in H^1 
\cap W^{1,4}$, even though the components $u_0$ 
and $v_0$ of $(u_0, v_0) \in \Sigma$ 
may not belong to $W^{1,4}$. Therefore, for our class of 
initial data $(u_0, v_0) \in \Sigma$, we can consider the 
Cauchy problem \eqref{ODE}-\eqref{id} in the same 
Banach space $E$ as in \cite{HQ21}.
For completeness and 
future reference, we present here the local well-posedness 
analysis of the ODE system. The proofs 
are found in Appendix \ref{AS4}.

Considering that the initial data $W(0,\xi)$ and $Z(0,\xi)$ 
take values in $[-\pi,\pi]$, and $q(0,\xi)=1$, it suffices 
to study the solutions of the Cauchy problem \eqref{ODE} 
and \eqref{id} within the following closed subset 
$\Omega\subset E$ (see the set $\Lambda$ 
in \cite[Section IV]{HQ21}):
\begin{equation}\label{Om}
	\begin{split}
		\Omega= & \Bigl\{
		(U,V,W,Z,q)\in E:
		\|U\|_{H^1\cap W^{1,4}},\|V\|_{H^1\cap W^{1,4}}
		\leq R_1,
		\,\,
		\|W\|_{L^2},\|Z\|_{L^2}\leq R_2,\\
		&
		\qquad\qquad\qquad\qquad\quad\,\,\,
		\|W\|_{L^{\infty}},\|Z\|_{L^{\infty}}
		\leq \frac{3\pi}{2},
		\,\,
		q^-\leq q(\xi)\leq q^+,
		\,\,\xi\in\mathbb{R}
		\Bigr\},
	\end{split}
\end{equation}
for some $R_1,R_2,q^-,q^+>0$.
Notice that
\begin{equation}\label{WZ-L4}
	\|W\|_{L^4}\leq\|W\|_{L^\infty}^{1/2}
	\|W\|_{L^2}^{1/2}\leq
	\left(\frac{3\pi}{2}R_2\right)^{1/2}\quad
	\mbox{and}
	\quad \|Z\|_{L^4}\leq
	\left(\frac{3\pi}{2}R_2\right)^{1/2},
\end{equation}
as soon as $W$ and $Z$ meet the 
conditions specified in \eqref{Om}.

To demonstrate the local existence and uniqueness 
of the solution to the ODE system 
\eqref{ODE} and \eqref{id}, it suffices to 
verify the following conditions:
\begin{enumerate}[(i)]
	\item	 For all $(U,V,W,Z,q)\in\Omega$, the right-hand 
	side of \eqref{ODE} belongs to the set $E$.
	
	\item The right-hand side of \eqref{ODE} is Lipschitz 
	continuous on the domain $\Omega$.
\end{enumerate}

We begin by proving a key technical lemma, which 
shows that the exponential term $\mathcal{E}$, which 
is part of the nonlocal terms $P_j$, $\px P_j$, $S_j$, 
and $\px S_j$, $j=1,2$ (as outlined in Proposition \ref{Pj}), 
decays exponentially quickly as $|\xi-\eta|\to\infty$.
\begin{lemma}[\cite{BC07},\cite{HQ21}]\label{L1}
	Introduce the function
	\begin{equation}\label{G}
		\Gamma(\xi)=\exp\left(
		\frac{q^-}{4}(R_2^2-|\xi|)
		\right),\quad \xi\in\mathbb{R}.
	\end{equation}
	Then the following estimate holds for 
	all $(U,V,W,Z,q)(t,\cdot)\in\Omega$:
	\begin{equation}\label{E-est}
		\mathcal{E}(t,\xi,\eta)\leq\Gamma(\xi-\eta),
		\quad \xi,\eta\in\mathbb{R},
	\end{equation}
	where $\mathcal{E}(t,\xi,\eta)$ 
	is given by \eqref{E}.
\end{lemma}

\begin{proof}
	The proof is given in Appendix \ref{AS4}.
\end{proof}

\begin{remark}
	From \eqref{G}, we obtain
	\begin{equation}\label{G-e}
		\|\Gamma\|_{L^1}=\frac{8}{q^-}
		\exp\left(\frac{q^-}{4}R_2^2\right),
	\end{equation}
	and
	\begin{equation}\label{mG-e}
		\|\xi\Gamma(\xi)\|_{L^1}=\frac{32}{(q^-)^2}
		\exp\left(\frac{q^-}{4}R_2^2\right).
	\end{equation}
\end{remark}

Applying Lemma \ref{L1}, we will show 
that the nonlocal terms in \eqref{ODE} 
are elements of the space $H^1\cap W^{1,4}$.

\begin{proposition}\label{POm}
	Suppose that $(U,V,W,Z,q)(t,\cdot)\in\Omega$.
	Then
	$$
	\left(P_j, \px P_j, S_j,
	\px S_j\right)(t,\cdot)
	\in \left(H^1\cap W^{1,4}\right)^4,\quad j=1,2, 
	$$
	where
	$P_j$, $\px P_j$, $S_j$ and $\px S_j$, $j=1,2$, 
	are given in Proposition \ref{Pj}.
\end{proposition}

\begin{proof}
	The proof is presented in Appendix \ref{AS4}.
\end{proof}

The above proposition suggests that the 
right-hand side of \eqref{ODE} maps, for every fixed $t$, 
an element from $\Omega$ to $E$, thereby proving item (i). 
The Lipschitz property (ii) of this map in $\Omega$ is 
established by the subsequent proposition.

\begin{proposition}\label{LipP}
	Let $P_{jk}, \px P_{jk}, S_{jk}$ and $\px S_{jk}$, 
	where $j,k=1,2$, be defined as $P_j, \px P_j, S_j$ 
	and $\px S_j$ in Proposition \ref{Pj}, respectively, but 
	with $(U_k,V_k,W_k,Z_k,q_k)$ instead of $(U,V,W,Z,q)$.
	Introduce (here and below, the argument $t$ is dropped):
	\begin{equation}\label{Nj}
		\begin{split}
			N_j=&\max\bigl\{
			\|P_{j1}-P_{j2}\|_{H^1\cap W^{1,4}},
			\|\px P_{j1}-\px P_{j2}\|_{H^1\cap W^{1,4}},\\
			&\qquad\,\,\,\|S_{j1}-S_{j2}\|_{H^1\cap W^{1,4}},
			\|\px S_{j1}-\px S_{j2}\|_{H^1\cap W^{1,4}}
			\bigr\},\quad j=1,2,
		\end{split}
	\end{equation}
	Then there exists a constant 
	$L>0$, such that for all 
	$(U_k,V_k,W_k,Z_k,q_k)(t,\cdot)\in\Omega$, $k=1,2$, 
	we have
	\begin{equation}\label{Pj-L}
		\begin{split}
			N_j\leq
			L\|(U_1,V_1,W_1,Z_1,q_1)-(U_2,V_2,W_2,Z_2,q_2)\|_E,
			\quad j=1,2.
		\end{split}
	\end{equation}
\end{proposition}

\begin{proof}
	The proof is given in Appendix \ref{AS4}.
\end{proof}

Finally, by employing the contraction 
principle, we deduce the following result 
(see 
\cite[Lemma 4.1]{HQ21} for a detailed proof):

\begin{theorem}[Local existence and 
	uniqueness of  \eqref{ODE}-\eqref{id}]\label{lwp}
	Suppose that $(u_0,v_0)\in\Sigma$, where 
	$(\Sigma,d_\Sigma)$ is defined by \eqref{Sig} and \eqref{d^2}.
	Then there exists a unique solution $(U,V,W,Z,q)$ 
	of the Cauchy problem \eqref{ODE}-\eqref{id} such that 
	\begin{equation*}
		(U,V,W,Z,q)\in C\left([-T,T],\Omega\right),
	\end{equation*}
	for some (small) $T>0$, where $\Omega$ 
	is given in \eqref{Om}.
\end{theorem}

\section{Global well-posedness of ODE system}\label{GS}

In order to extend the local solution of the ODE 
system \eqref{ODE}-\eqref{id}, obtained in 
Theorem \ref{lwp}, one must show that the 
norms (see \eqref{Om}, particularly that 
$q^-\leq q(\xi)\leq q^+$, $q^-,q^+>0$, and Remark \ref{Rpi})
\begin{equation*}
	\|U\|_{H^1\cap W^{1,4}},
	\|V\|_{H^1\cap W^{1,4}},
	\|W\|_{L^2\cap L^\infty},
	\|Z\|_{L^2\cap L^\infty},
	\|q\|_{L^\infty},
	\|q^{-1}\|_{L^\infty},
\end{equation*}
are uniformly bounded. 
The proof of this fact is presented in \cite[Section V]{HQ21}. 

In the remainder of this section, we 
gather a collection of a priori bounds that will be 
utilized in subsequent sections.
The following conservation laws \cite[Equations 
(5.3)--(5.5)]{HQ21} play a crucial role in 
the derivation of these bounds (compare with Proposition \ref{PCL}).
\begin{equation}\label{cn}
	\begin{split}
		&\int_{-\infty}^{\infty}
		\left(U^2\cos^2\frac{W}{2}+
		\sin^2\frac{W}{2}\right)(t,\xi)
		\left(
		q\cos^2\frac{Z}{2}
		\right)(t,\xi)\,d\xi = E_{u_0},\\
		&\int_{-\infty}^{\infty}
		\left(V^2\cos^2\frac{Z}{2}+
		\sin^2\frac{Z}{2}\right)(t,\xi)
		\left(
		q\cos^2\frac{W}{2}
		\right)(t,\xi)\,d\xi = E_{v_0},
	\end{split}
\end{equation}
and
\begin{equation}\label{gcn}
	\int_{-\infty}^{\infty}\left(
	qUV\cos^2\frac{W}{2}\cos^2\frac{Z}{2}
	+\frac{q}{4}\sin W\sin Z
	\right)(t,\xi)\,d\xi=G_0,
\end{equation}
as well as
\begin{equation}\label{hcn}
	\begin{split}
		\int_{-\infty}^{\infty}
		&\left(3U^2V^2\cos^2\frac{W}{2}\cos^2\frac{Z}{2}
		+U^2\cos^2\frac{W}{2}\sin^2\frac{Z}{2}
		+V^2\sin^2\frac{W}{2}\cos^2\frac{Z}{2}\right.\\
		&\left.\quad
		+UV\sin W\sin Z-\sin^2\frac{W}{2}\sin^2\frac{Z}{2}\right)
		(t,\xi)q(t,\xi)\,d\xi= H_0,
	\end{split}
\end{equation}
for any $t\in[-T,T]$, where the constants 
$E_{u_0},E_{v_0},G_0$ and $H_0$ are given in \eqref{consq}.

In the proposition presented below, 
we provide slightly different estimates compared 
to those in \cite{HQ21}. These estimates are then 
utilized to prove Theorem \ref{Thm}, specifically 
item (4), as discussed in Section \ref{sgs}.
\begin{proposition}
Assume that $(U,V,W,Z,q)\in C\left([-T,T], \Omega\right)$
is a local solution of \eqref{ODE}-\eqref{id} with 
$(u_0,v_0)\in\Sigma$, where $\Sigma$ is defined in \eqref{Sig}.
Then the following a priori estimates hold:
\begin{equation}\label{apest1}
	\|U(t,\cdot)\|_{L^\infty}\leq 
	E_{u_0}^{1/2},\quad
	\|V(t,\cdot)\|_{L^\infty}\leq 
	E_{v_0}^{1/2},
\end{equation}
and
\begin{equation}\label{u_x^2n}
	\int_{-\infty}^{\infty}
	q\sin^2\frac{W}{2}
	\sin^2\frac{Z}{2}\,d\xi
	\leq 7E_{u_0}E_{v_0}-H_0.
\end{equation}
\end{proposition}
\begin{proof}
	\textbf{Step 1, estimates \eqref{apest1}.}
	Using that (see \cite[Equation (5.1)]{HQ21})
	\begin{equation}\label{pxiUV}
		\pxi U=\frac{q}{2}\sin W\cos^2\frac{Z}{2}\quad
		\mbox{and}\quad
		\pxi V=\frac{q}{2}\cos^2\frac{W}{2}\sin Z,
	\end{equation}
	and \eqref{cn}, we 
	have, cf.\,\,\cite[Equation (5.7)]{HQ21}
	(here and below the $t$-argument is dropped)
	\begin{equation*}
		\begin{split}
			U^2(\xi)=&2\int_{-\infty}^\xi U(\eta)
			\peta U(\eta)\,d\eta
			\leq\int_{-\infty}^\infty\left(q|U\sin W|
			\cos^2\frac{Z}{2}\right)(\eta)\,d\eta\\
			&=2\int_{-\infty}^\infty
			\left(q^{1/2}
			\left|U\cos\frac{W}{2}\cos\frac{Z}{2}\right|
			\right)(\eta)\cdot
			\left(q^{1/2}
			\left|
			\sin\frac{W}{2}
			\cos\frac{Z}{2}
			\right|
			\right)(\eta)\,d\eta\\
			&\leq\int_{-\infty}^\infty
			\left(U^2\cos^2\frac{W}{2}+
			\sin^2\frac{W}{2}\right)(\eta)
			\left(
			q\cos^2\frac{Z}{2}
			\right)(\eta)\,d\eta
			=E_{u_0},
		\end{split}
	\end{equation*}
	which implies \eqref{apest1} for $\|U\|_{L^\infty}$.
	The proof for $\|V\|_{L^\infty}$ is similar.
	
	\medskip
	
	\textbf{Step 2, estimate \eqref{u_x^2n}.}
	Recalling the conservation law \eqref{hcn}, 
	we have (see \eqref{u_x^2})
	\begin{equation*}
		\begin{split}
			\int_{-\infty}^{\infty}
			q\sin^2\frac{W}{2}&
			\sin^2\frac{Z}{2}\,d\xi
			=\int_{-\infty}^{\infty}
			\left(3U^2V^2\cos^2\frac{W}{2}\cos^2\frac{Z}{2}
			+U^2\cos^2\frac{W}{2}\sin^2\frac{Z}{2}\right.\\
			&\left.\qquad\qquad\qquad\qquad\,\,
			+V^2\sin^2\frac{W}{2}\cos^2\frac{Z}{2}
			+UV\sin W\sin Z\right)q\,d\xi-H_0\\
			&\leq\|U\|^2_{L^\infty}E_{v_0}
			+\int_{-\infty}^{\infty}
			\left(2U^2V^2\cos^2\frac{W}{2}\cos^2\frac{Z}{2}
			+V^2\sin^2\frac{W}{2}\cos^2\frac{Z}{2}\right.\\
			&\left.\qquad\qquad\qquad\qquad
			\qquad
			+UV\sin W\sin Z\right)q\,d\xi-H_0\\
			&\leq 2\|U\|^2_{L^\infty}E_{v_0}
			+\|V\|^2_{L^\infty}E_{u_0}
			+4\|UV\|_{L^\infty}
			E_{u_0}^{1/2}E_{v_0}^{1/2}-H_0
			\leq 7E_{u_0}E_{v_0}-H_0,
		\end{split}
	\end{equation*}
	where we have used that
	$$
	\int_{-\infty}^{\infty}qUV\sin W\sin Z\,d\xi
	\leq 4\|UV\|_{L^\infty}
	\left\|q^{1/2}\sin\frac{W}{2}\cos\frac{Z}{2}
	\right\|_{L^2}
	\left\|
	q^{1/2}\cos\frac{W}{2}\sin\frac{Z}{2}
	\right\|_{L^2},
	$$
	by the Cauchy-Schwarz inequality.
\end{proof}

The following uniform bounds for the nonlocal 
terms are proved in \cite[p.~13--14]{HQ21}
(see also Section \ref{cls} herein):
\begin{equation}\label{PjSj}
	\|P_j(t,\cdot)\|_{L^p},
	\|\px P_j(t,\cdot)\|_{L^p},
	\|S_j(t,\cdot)\|_{L^p},
	\|\px S_j(t,\cdot)\|_{L^p}
	\leq C_p,\quad
	j=1,2,
\end{equation}
for some constant $C_p=C_p(E_{u_0},E_{v_0}, H_0)>0$, 
$p\in [1,\infty]$. Then from the fifth equation 
in \eqref{ODE} we conclude that
(see \cite[p.~13]{HQ21} and 
\cite[Equation (4.29)]{BC07})
\begin{equation}\label{qap}
	\kappa_0^{-1}\leq\|q(t,\cdot)\|_{L^\infty}
	\leq\kappa_0,\quad 
	\kappa_0=e^{CT},
	\qquad t\in[-T,T], 
\end{equation}
for some constant $C=C(E_{u_0},E_{v_0}, H_0)>0$. 
Moreover, we have the following uniform bounds 
for $U,V$, $W$, $Z$ for 
all $t\in[-T,T]$ (see \cite[p.~14--15]{HQ21}):
\begin{equation}\label{UVe}
	\begin{split}
		&\|U(t,\cdot)\|_{H^1\cap W^{1,4}}\leq
		\|U_0\|_{H^1\cap W^{1,4}}+CT,\\
		&\|V(t,\cdot)\|_{H^1\cap W^{1,4}}\leq
		\|V_0\|_{H^1\cap W^{1,4}}+CT,
	\end{split}
\end{equation}
and, for $p=2,\infty$,
\begin{equation}\label{WZe}
	\begin{split}
		&\|W(t,\cdot)\|_{L^p}\leq\|W_0\|_{L^p}+CT,\\
		&\|Z(t,\cdot)\|_{L^p}\leq\|Z_0\|_{L^p}+CT,
	\end{split}
\end{equation}
for some constant $C=C(E_{u_0},E_{v_0},H_0)>0$.

Combining a priori estimates 
\eqref{qap}, \eqref{UVe} and \eqref{WZe} 
we arrive at the following result:
\begin{theorem}
[Global existence and uniqueness of \eqref{ODE}-\eqref{id}]
\label{gwp}
Suppose that $(u_0,v_0)\in\Sigma$, where $\Sigma$ is defined 
in \eqref{Sig}. Then there exists a unique global 
solution $(U,V,W,Z,q)(t,\xi)$ of the Cauchy 
problem \eqref{ODE}-\eqref{id} such that 
\begin{equation*}
	(U,V,W,Z,q)\in C\left(
	[-T,T], \Omega
	\right),\quad
	\mbox{for any $T>0$},
\end{equation*}
where $\Omega$ is given in \eqref{Om}.
\end{theorem}

\section{Global solutions of two-component system}\label{GsN}

In this section, we demonstrate the existence 
of a global conservative solution to the Cauchy 
problem for the equation \eqref{t-c-N}, as defined 
in Definitions \ref{defs} and \ref{defsc}, by employing 
the global solution $(U,V,W,Z,q)(t,\xi)$ of the corresponding 
ODE system. Building upon the work of \cite{HQ21}, we identify new 
conservation properties of the solution $(u,v)$ and introduce 
a novel continuity property for the data-to-solution map, as detailed 
in Propositions \ref{mcl} and \ref{Pdsm} below. Additionally, we establish 
the existence of a global weak solution $(u,v)(t,\cdot)\in\Sigma$ 
with initial data $(u_0,v_0)$ from the 
same set $\Sigma$ (see \eqref{Sig} and \eqref{d^2}), 
contrasting with \cite{HQ21}, where 
it was assumed that $u_0,v_0\in H^1\cap W^{1,4}$. 
This advancement enables us to 
construct a global \textit{semigroup} of conservative 
solutions, as will be discussed in 
Section \ref{sgs} below.

\subsection{Reconstruction of $(u,v)$ in terms of $(U,V,y)$}\label{ds}

Define the characteristic $y(t,\xi)$ as follows (cf.\,\,\eqref{char}):
\begin{equation}\label{char1}
	y(t,\xi)=y_0(\xi)
	+\int_0^t\left(UV\right)(\tau,\xi)\,d\tau,
	\quad t,\xi\in\mathbb{R},
\end{equation}
where $y_0(\xi)$ is given by \eqref{y0} and $U,V$ 
are the global solutions of the ODE system \eqref{ODE}-\eqref{id} 
given in Theorem \ref{gwp}. 
We present a concise summary of certain properties 
of the function $y(t,\xi)$ in the following proposition 
(see \cite[Section VI, Steps 1-3]{HQ21}):

\begin{proposition}\label{propy}
The characteristic $y(t,\xi)$ defined in \eqref{char1} 
satisfies the following properties:
\begin{enumerate}
	\item the image of $y(t,\cdot)$ is $\mathbb{R}$,
	\item $y(t,\cdot)$ is nondecreasing,
	\item if $\xi<\eta$ and $y(t,\xi)=y(t,\eta)$, then
	$U(t,\xi)=U(t,\eta)$ and $V(t,\xi)=V(t,\eta)$,
\end{enumerate}
for all fixed $t\in\mathbb{R}$.
\end{proposition}

\begin{proof}
Using \eqref{apest1}, we observe that
\begin{equation*}
	y_0(\xi)-E_{u_0}^{1/2}E_{v_0}^{1/2}|t|
	\leq y(t,\xi)
	\leq y_0(\xi)+E_{u_0}^{1/2}E_{v_0}^{1/2}|t|,
\end{equation*}
which implies that $y(t,\xi)\to\pm\infty$ as 
$\xi\to\pm\infty$ for all $t\in\mathbb{R}$ 
and item (1) is proved.
	
Recalling the formal identities \eqref{q} and \eqref{trig},
one has (see \cite[Section VI, Step 2]{HQ21}):
\begin{equation}\label{pxi-y}
	\pxi y=q\cos^2\frac{W}{2}\cos^2\frac{Z}{2}.
\end{equation}
Since $q>0$, item (2) follows from \eqref{pxi-y}.
	
Finally, let us consider $\xi< \eta$ such 
that $y(t,\xi)=y(t, \eta)$. 
Using \eqref{pxi-y}, we can conclude that
(the argument $t$ is omitted for clarity)
\begin{equation*}
	0=\int_{\xi}^{\eta}\pxi y(s)\,ds
	=\int_{\xi}^{\eta}
	\left(q\cos^2\frac{W}{2}\cos^2\frac{Z}{2}\right)(s)\,ds.
\end{equation*}
which implies that for every $s\in[\xi,\eta]$ either 
$W(s)=\pi+2\pi n$ or $Z(s)=\pi+2\pi n$, 
where $n\in\mathbb{Z}$. 
Consequently, from \eqref{pxiUV}, we have
\begin{equation*}
	\begin{split}
		&U(\eta)-U(\xi)
		=\int_{\xi}^{\eta}\left(\frac{q}{2}
		\sin W\cos^2\frac{Z}{2}
		\right)(s)\,ds=0,\\
		&V(\eta)-V(\xi)
		=\int_{\xi}^{\eta}\left(\frac{q}{2}
		\cos^2\frac{W}{2}\sin Z
		\right)(s)\,ds=0,
	\end{split}
\end{equation*}
and item (3) is established.
\end{proof}

We define $(u,v)(t,x)$ as follows (cf.~\eqref{UV}):
\begin{align}\label{uvdef}
	u(t,x)=U(t,\xi),\quad v(t,x)=V(t,\xi),\quad
	\mbox{if }x=y(t,\xi),\quad t\in\mathbb{R}.
\end{align}
As per Proposition \ref{propy}, this 
definition is valid.
Observe that equations \eqref{pxi-y}, 
\eqref{uvdef} and \eqref{pxiUV} imply
\begin{equation}\label{pxuv}
	\begin{split}
		&\px u(t,x)=\frac{\pxi U(t,\xi)}{\pxi y(t,\xi)}
		=\tan\frac{W(t,\xi)}{2},\quad x=y(t,\xi),
		\quad\xi\in\mathbb{R}\setminus D_W(t),\\
		&\px v(t,x)=\frac{\pxi V(t,\xi)}{\pxi y(t,\xi)}
		=\tan\frac{Z(t,\xi)}{2},\quad\,\, x=y(t,\xi),
		\quad\xi\in\mathbb{R}\setminus D_Z(t),
	\end{split}
\end{equation}
where the sets $D_W(t)$ and $D_Z(t)$ 
are defined in \eqref{DN}.

\subsection{Global weak solution}

We begin with the following lemma 
(see \cite[Section VI, Step 4]{HQ21}):

\begin{lemma}\label{lHQ21}
Suppose $(u_0, v_0) \in \Sigma$, 
where $\Sigma$ is defined in \eqref{Sig}. 
Then, for a.e.~$t \in N_W$ and a.e.~$\xi \in D_{W}(t)$, 
we have $V(t, \xi) = 0$. Furthermore, 
for a.e.~$t \in N_Z$ and a.e.~$\xi 
\in D_{Z}(t)$, $U(t, \xi) = 0$.
\end{lemma}

\begin{proof}
Introduce the sets (cf.~\eqref{DN})
\begin{equation*}
	\begin{split}
		&D_{V,W}^{(j)}(t)
		=\{\xi\in D_W(t): (-1)^jV(t,\xi)>0\},
		\quad j=1,2,\\
		&D_{U,Z}^{(j)}(t)
		=\{\xi\in D_Z(t): (-1)^jU(t,\xi)>0\},
		\quad j=1,2,
	\end{split}
\end{equation*}
and
\begin{equation*}
	\begin{split}
		& N_{V,W}^{(j)}=\left\{t\in N_W:
		\mathrm{meas}\left(D_{V,W}^{(j)}(t)\right)>0\right\},
		\quad j=1,2,\\
		& N_{U,Z}^{(j)}=\left\{t\in N_Z:
		\mathrm{meas}\left(D_{U,Z}^{(j)}(t)
		\right)>0\right\},\quad j=1,2.
	\end{split}
\end{equation*}
We must prove that
\begin{equation}\label{NzNw}
	\mathrm{meas}\left(N_{U,Z}^{(j)}\right)=
	\mathrm{meas}\left(N_{V,W}^{(j)}\right)=0,
	\quad j=1,2.
\end{equation}
Consider the set $N_{V,W}^{(1)}$.
Suppose to the contrary that
$\mathrm{meas}\bigl(N_{V,W}^{(1)}\bigr)>0$.
Taking into account that (i) $W(\cdot,\xi)$ 
is absolutely continuous for 
a.e.~$\xi\in\mathbb{R}$ (this follows from \eqref{ODE}) 
and (ii) $W(t,\xi)=\pi$ for all 
$(t,\xi)\in N_{V,W}^{(1)}\times D_{V,W}^{(1)}(t)$, we have 
$\pt W(t,\xi)=0$ a.e. in
$N_{V,W}^{(1)}\times D_{V,W}^{(1)}(t)$. 
On the other hand, from \eqref{ODE} we conclude that 
$\pt W(t,\xi)=-V(t,\xi)$ for all 
$t\in N_{V,W}^{(1)}$ and for a.e. $\xi\in D_{V,W}^{(1)}(t)$.
Since $V(t,\xi)<0$ for such $(t,\xi)$ and the Lebesgue measure 
on $\mathbb{R}^2$ of the set 
$N_{V,W}^{(1)}\times D_{V,W}^{(1)}(t)$ is 
positive, we arrive at a contradiction.
Treating $N_{V,W}^{(2)}$ and $N_{U,Z}^{(j)}$, where $j=1,2$, 
in the same manner, we arrive at \eqref{NzNw}.
\end{proof}

For later use, let us gather the following properties 
(see \cite[Section VI, Steps 5–6]{HQ21}).

\begin{proposition}\label{Pgws}
The functions $u=u(t,x)$ and $v=v(t,x)$, 
which are defined by \eqref{uvdef}, 
possess the following properties:
\begin{enumerate}
	\item the uniform bounds given in 
	Definition \ref{defsc}, item (3);
	
	\item the Lipschitz property \eqref{uL};
	
	\item  the H\"older continuity property \eqref{uH};

	\item the conservation laws given in Definition \ref{defsc}, 
	item (1) and Definition \ref{defsc}, item (4).
\end{enumerate}
\end{proposition}

\begin{proof}
The proof is presented in Appendix \ref{AS6}.
\end{proof}

Next, we have that $(u,v)$ satisfies the 
weak formulations \eqref{uvw}:

\begin{proposition}\label{wsol}
The vector function $(u,v)(t,x)$ defined by 
\eqref{uvdef} satisfies 
the equations \eqref{uvw}.
\end{proposition}

\begin{proof}
A detailed argument is presented 
in \cite[Section VI, Step 7]{HQ21}.
\end{proof}

In the proposition presented below, we 
introduce the Radon measures $\lambda^{(u)}_t$, 
$\lambda^{(v)}_t$, and $\lambda^{(uv)}_t$
which enable us to extend the conservation laws 
from Definition \ref{defsc}, 
item (3), to \textit{all} $t\in\R$.

\begin{proposition}\label{mcl}
Assume that $(u_0,v_0)\in\Sigma$, 
where $\Sigma$ is defined in \eqref{Sig}.
Define the Radon measures 
\begin{equation}\label{rm}
	\begin{split}
		&\lambda^{(u)}_t\left([a,b]\right)
		=\int_{\{\xi:y(t,\xi)\in[a,b]\}}
		\left(q\sin^2\frac{W}{2}
		\cos^2\frac{Z}{2}\right)(t,\xi)\,d\xi,
		\\
		&\lambda^{(v)}_t\left([a,b]\right)
		=\int_{\{\xi:y(t,\xi)\in[a,b]\}}
		\left(q\cos^2\frac{W}{2}
		\sin^2\frac{Z}{2}\right)(t,\xi)\,d\xi,
		\\
		&\lambda^{(uv)}_t\left([a,b]\right)
		=\int_{\{\xi:y(t,\xi)\in[a,b]\}}
		\left(q\sin^2\frac{W}{2}
		\sin^2\frac{Z}{2}\right)(t,\xi)
		\,d\xi,
	\end{split}
\end{equation}
where $(U,V,W,Z,q)$ represents the unique 
global solution of the ODE system \eqref{ODE}-\eqref{id} 
as established in Theorem \ref{gwp}. The properties of 
$\lambda^{(u)}_t$, $\lambda^{(v)}_t$ and $\lambda^{(uv)}_t$
are as follows:
\begin{enumerate}
	\item the absolute continuous and singular parts 
	meet the specifications in Definition \ref{defsc}, item (2)(a);
	
	\item the conservation laws in 
	Definition \ref{defsc}, item (2)(b) hold;
	
	\item inequality \eqref{csho} is valid;

	\item the transport equations, along with the 
	source terms described in Definition \ref{defsc}, 
	specifically item (5), are valid.
\end{enumerate}
\end{proposition}

\begin{proof}
Let us prove item (1) for $\lambda_t^{(u)}$.
Introducing set 
$
\mathcal{Y}_t(a,b)=
\{\xi:y(t,\xi)\in[a,b]\},
$
we obtain from \eqref{rm}, \eqref{pxi-y} and
\eqref{pxuv} that
\begin{equation}\label{lam-u}
	\lambda_t^{(u)}\left([a,b]\right)
	=\int_a^b(\px u)^2(t,x)\,dx
	+\int_{\mathcal{Y}_t(a,b)\cap D_W(t)}
	\left(q
	\cos^2\frac{Z}{2}\right)(t,\xi)\,d\xi,
\end{equation}
recalling the definition 
\eqref{DN} of $D_W(t)$. Equation \eqref{lam-u} implies that 
$d\lambda_t^{(u,ac)}=(\px u)^2\,dx$ and, in view of 
Lemma \ref{lHQ21}, the nonzero singular part is supported 
on the set where $v(t,\cdot)=0$ for a.e.\,\,$t\in\R$.
The proof of item (1) for $\lambda_t^{(v)}$ and 
$\lambda_t^{(uv)}$ is similar.

To prove item (2) we notice that
\begin{equation}\label{uux}
\begin{split}
	&\int_{-\infty}^\infty\left(
	qU^2\cos^2\frac{W}{2}\sin^2\frac{Z}{2}
	\right)(t,\xi)\,d\xi
	=\int_{-\infty}^\infty
	u^2(t,x)\,d\lambda^{(v)}_t,\\
	&\int_{-\infty}^\infty\left(
	qV^2\sin^2\frac{W}{2}\cos^2\frac{Z}{2}
	\right)(t,\xi)\,d\xi
	=\int_{-\infty}^\infty
	v^2(t,x)\,d\lambda^{(u)}_t,
\end{split}
\end{equation}
where the variable $\xi$ has been transformed 
to $x = y(t, \xi)$ and we have used \eqref{rm}.
Now item (2) follows from \eqref{cn}, \eqref{hcn}, 
\eqref{uvdef}, \eqref{pxuv}, and \eqref{uux}.

Item (3) follows from \eqref{uux}, the inequality
\begin{equation*}
	\begin{split}
		-\int_{-\infty}^{\infty}
		\left(q\sin^2\frac{W}{2}\sin^2\frac{Z}{2}
		\right)(t,\xi)\,d\xi
		\leq-\int_{\mathbb{R}\setminus(D_W(t)\cup D_Z(t))}
		\left(q\sin^2\frac{W}{2}\sin^2\frac{Z}{2}
		\right)(t,\xi)\,d\xi,
	\end{split}
\end{equation*}
for any $t\in\R$,
and the identity
\begin{equation*}
	\int_{-\infty}^{\infty}
	\tilde{A}_1(t,\xi)\,d\xi=
	\int_{\mathbb{R}\setminus(D_W(t)\cup D_Z(t))}
	\tilde{A}_1(t,\xi)\,d\xi,
\end{equation*}
for any $t\in\R$, where 
$$
\tilde{A}_1=q\left(3U^2V^2\cos^2\frac{W}{2}\cos^2\frac{Z}{2}
+UV\sin W\sin Z\right).
$$

It remains to prove item (4). 
Let us establish the transport equation 
for $\lambda^{(u)}_t$. For any 
$\Phi_U(t,\xi)=\phi_u(t,y(t,\xi))$, 
where $\phi_u$ is a test function 
in \eqref{uvw}, we have (the arguments 
$t,\xi,x$ are dropped here)
\begin{equation}\label{tI}
	\tilde{I}:=\int_{-T}^T\int_{-\infty}^{\infty}
	\left(q\sin^2\frac{W}{2}\cos^2\frac{Z}{2}\right)
	\pt\Phi_U\,d\xi\,dt=
	\int_{-T}^T\int_{-\infty}^{\infty}
	(\pt\phi_u+uv\px\phi_u)
	\, d\lambda^{(u)}_t\,dt,
\end{equation}
where we have changed the 
variables: $x=y(t,\xi)$, 
used \eqref{rm} and that 
$\pt\Phi_U=\pt\phi_u+uv\px\phi_u$.

On the other hand, using
\begin{equation}\label{cn-t2}
	\pt\left(
	q\sin^2\frac{W}{2}\cos^2\frac{Z}{2}
	\right)=
	q\sin W\cos^2\frac{Z}{2}\left(
	U^2V-P_1-\px P_2
	\right)
	+\frac{q}{2}U\sin^2\frac{W}{2}\sin Z,
\end{equation}
and
integrating by parts,
we obtain (here and below, 
we omit the argument $t$ from $D_W$ and $D_Z$):
\begin{equation*}
	\begin{split}
		\tilde{I}&=
		-\int_{-T}^T\int_{\mathbb{R}\setminus\{D_W\cup D_Z\}}
		\left(q\sin W\cos^2\frac{Z}{2}\left(
		U^2V-P_1-\px P_2
		\right)
		+\frac{q}{2}U\sin^2\frac{W}{2}
		\sin Z\right)\Phi_U\,d\xi\,dt
		\\ & \quad \, 
		-\int_{-T}^T\int_{D_W}
		\left(\frac{q}{2}U\sin Z\right)\Phi_U\,d\xi\,dt.
	\end{split}
\end{equation*}
Assuming that $\mathrm{meas}(N_W)=0$ and changing 
the variables by setting $x=y(t,\xi)$, 
we can conclude that
\begin{equation*}
	\tilde{I}=
	-\int_{-T}^T\int_{-\infty}^{\infty}
	\left(
	(2\px u)(u^2v-P_1-\px P_2)+u(\px u)^2\px v
	\right)\phi_{u}\,dx\,dt,
\end{equation*}
which, combined with \eqref{tI}, leads us to 
derive the desired transport equation for the measure 
$\lambda^{(u)}_t$ (see Remark \ref{wms} for additional context). 
Similarly, the measure $\lambda^{(v)}_t$ can be 
treated using the same method.

Finally, we derive the equation for $\lambda^{(uv)}_t$. 
Considering that
\begin{equation}\label{cn-t3}
	\begin{split}
		\pt\left(q\sin^2\frac{W}{2}\sin^2\frac{Z}{2}\right)
		=&q\sin W\sin^2\frac{Z}{2}
		\left(U^2V-(P_1+\px P_2)\right)\\
		&+q\sin^2\frac{W}{2}\sin Z
		\left(UV^2-(S_1+\px S_2)\right),
	\end{split}
\end{equation}
we obtain
\begin{equation*}
	\begin{split}
		&\int_{-T}^T\int_{-\infty}^{\infty}
		\left(q\sin^2\frac{W}{2}\sin^2\frac{Z}{2}\right)
		\pt\Phi_{U,V}\,d\xi\,dt\\
		&=-\int_{-T}^T
		\int_{\mathbb{R}\setminus\{D_W\cup D_Z\}}
		\pt\left(q\sin^2\frac{W}{2}\sin^2\frac{Z}{2}
		\right)\Phi_{U,V}\,d\xi\,dt\\
		&\quad\,+\int_{-T}^T\int_{D_Z}
		\left(q(P_1+\px P_2)
		\sin W\right)\Phi_{U,V}\,d\xi\,dt
		+\int_{-T}^T\int_{D_W}
		\left(q(S_1+\px S_2)
		\sin Z\right)\Phi_{U,V}\,d\xi\,dt.
	\end{split}
\end{equation*}
Assuming that the measure of both sets $N_W$ and $N_Z$ 
is zero: $\mathrm{meas}(N_W)=\mathrm{meas}(N_Z)=0$, 
and by performing a change of 
variables where $x = y(t, \xi)$, we 
derive the transport equation for $\lambda^{(uv)}_t$.

\end{proof}

In the proposition below, we demonstrate the 
continuous dependence of the 
solution on the initial data, 
measured in the uniform norm.

\begin{proposition}\label{Pdsm}
Assume that $d_\Sigma\left((u_{0,n},v_{0,n}),
(u_0,v_0)\right)\to0$ as $n\to\infty$,
where the metric space $(\Sigma,d_\Sigma)$ is 
defined by \eqref{Sig} and \eqref{d^2}.
Then we have
\begin{equation}\label{dsm}
	\left(
	\|(u_n-u)\|_{L^\infty([-T,T]\times\mathbb{R})}
	+\|(v_n-v)\|_{L^\infty([-T,T]\times\mathbb{R})}
	\right)\to 0,
	\quad n\to\infty,
\end{equation}
for all $T>0$. Here, $(u_n,v_n)$ are the solutions 
of \eqref{t-c-N-n} that correspond to 
the initial data $(u_{0,n},v_{0,n})$, 
which are defined using the same 
procedure as $(u,v)$ in Section \ref{ds}.
\end{proposition}

\begin{proof}
\textbf{Step 1.} Let us show that
\begin{equation}\label{y0y0n}
	\|y_{0,n}-y_0\|_{L^\infty}\to 0,\quad
	n\to\infty,
\end{equation}
where $y_0(\xi)$ is given in \eqref{y0}, while
$y_{0,n}(\xi)$ is defined as follows:
\begin{equation}\label{y0n}
	\int_0^{y_{0,n}(\xi)}
	\left(1+(\px u_{0,n})^2(x)\right)
	\left(1+(\px v_{0,n})^2(x)\right)dx=\xi.
\end{equation}
Consequently, $y_0(\xi)$ and 
$y_{0,n}(\xi)$ are defined implicitly 
by integrals involving $F = (1 + (\partial_x u_0)^2)
(1 + (\partial_x v_0)^2)$ and 
$F_n = (1 + (\partial_x u_{0,n})^2)
(1 + (\partial_x v_{0,n})^2)$, respectively. 
Subtracting these integral definitions gives
$$
\int_{y_{0,n}}^{y_0} F(x) \, dx = \int_0^{y_{0,n}} 
(F_n(x) - F(x)) \, dx.
$$
Using $F \geq 1$, the interval length 
$|y_{0,n}(\xi) - y_0(\xi)|$ can be controlled by 
the integral of $|F_n - F|$, leading to
$$
|y_{0,n}(\xi) - y_0(\xi)| \leq \int_{-\infty}^\infty 
\left| F_n(x) - F(x) \right| \, dx.
$$
Expanding the product terms in $F_n$ and $F$ gives
\begin{equation}\label{y0y0n1}
	\begin{split}
		|y_{0,n}(\xi)-y_0(\xi)|\leq
		\int_{-\infty}^{\infty}
		&\left(
		\left|(\px u_{0,n})^2-(\px u_{0})^2\right|
		+
		\left|(\px v_{0,n})^2-(\px v_{0})^2\right|
		\right.\\
		&\,\,\left.+
		\left|((\px u_{0,n})\px v_{0,n})^2
		-((\px u_{0})\px v_{0})^2\right|
		\right)(x)\,dx,
	\end{split}
\end{equation}
for all $\xi\in\mathbb{R}$.

Using that $\px u_{0,n}\to\px u_0$,
$\px v_{0,n}\to\px v_0$ and 
$(\px u_{0,n})\px v_{0,n}\to
(\px u_{0})\px v_{0}$
in $L^2$ (see \eqref{d^2}), the  
estimate \eqref{y0y0n1} 
implies \eqref{y0y0n}.

\medskip
	
\textbf{Step 2.}
Let us prove that
\begin{equation}\label{U0U0n}
	\left\|
	U_{0,n}-U_0
	\right\|_{L^\infty}
	\to 0,\,\,
	\left\|
	V_{0,n}-V_0
	\right\|_{L^\infty}
	\to 0,\quad n\to\infty.
\end{equation}
Observing that
\begin{equation*}
	\begin{split}
		\sup\limits_{\xi\in\mathbb{R}}\left|
		U_{0,n}(\xi)-U_0(\xi)
		\right|
		&=\sup\limits_{\xi\in\mathbb{R}}\left|
		u_{0,n}(y_{0,n}(\xi))-u_0(y_0(\xi))
		\right|\\
		&\leq
		\sup\limits_{\xi\in\mathbb{R}}\left|
		u_{0,n}(y_{0,n}(\xi))-u_0(y_{0,n}(\xi))
		\right|+
		\sup\limits_{\xi\in\mathbb{R}}\left|
		u_{0}(y_{0,n}(\xi))-u_0(y_{0}(\xi))
		\right|\\
		&\leq
		\sup\limits_{x\in\mathbb{R}}\left|
		u_{0,n}(x)-u_0(x)\right|
		+\sup\limits_{\xi\in\mathbb{R}}
		|y_{0,n}(\xi)-y_0(\xi)|^{3/4},
	\end{split}
\end{equation*}
and using the Sobolev embedding theorem 
$(H^1\hookrightarrow L^\infty)$
as well as \eqref{y0y0n}, we obtain 
\eqref{U0U0n} for $U_{0,n}$.
The convergence claim for $V_{0,n}$ in \eqref{U0U0n} 
can be proven similarly.

\medskip
	
\textbf{Step 3.} We will establish 
the following two limits:
\begin{equation}\label{W0W0n}
	\|W_{0,n}-W_0\|_{L^2},\,\,
	\|Z_{0,n}-Z_0\|_{L^2}\to 0,
	\quad \text{as $n\to\infty$}.
\end{equation}
Observe that
\begin{equation}\label{W0Wn}
	\begin{split}
		&\left\|W_{0,n}-W_0\right\|^2_{L^2}
		= 2\int_{-\infty}^{\infty}
		\left(
		\arctan(\px u_{0,n})(y_{0,n}(\xi))-
		\arctan(\px u_{0})(y_{0}(\xi))
		\right)^2d\xi\\
		&\qquad\qquad\leq
		4\int_{-\infty}^{\infty}
		\left(
		\arctan(\px u_{0,n})(y_{0,n}(\xi))-
		\arctan(\px u_{0})(y_{0,n}(\xi))
		\right)^2d\xi\\
		&\qquad\qquad\quad+
		4\int_{-\infty}^{\infty}
		\left(
		\arctan(\px u_{0})(y_{0,n}(\xi))-
		\arctan(\px u_{0})(y_{0}(\xi))
		\right)^2d\xi=: 4(I_{3,n}+I_{4,n}).
	\end{split}
\end{equation}
By applying the change of variables $x = y_{0,n}(\xi)$ 
to the integral $I_{3,n}$ and utilizing the 
relationship given in \eqref{y0n}, 
we derive the following result:
\begin{equation*}
	\frac{dy_{0,n}}{d\xi}
	\left(1+(\px u_{0,n})^2(y_{0,n})\right)
	\left(1+(\px v_{0,n})^2(y_{0,n})\right)=1,
	\qquad x=y_{0,n},
\end{equation*}
we obtain
\begin{equation*}
	\begin{split}
		I_{3,n}&=
		\int_{-\infty}^{\infty}
		\left(\left(\arctan(\px u_{0,n})-
		\arctan(\px u_{0})\right)^2
		\left(1+(\px u_{0,n})^2\right)
		\left(1+(\px v_{0,n})^2\right)\right)(x)\,dx
		\\
		&\leq
		\int_{-\infty}^{\infty}
		\left(\left(\arctan(\px u_{0,n})-
		\arctan(\px u_{0})\right)^2
		\left((\px u_{0,n})^2
		+(\px v_{0,n})^2
		+(\px u_{0,n})^2(\px v_{0,n})^2
		\right)\right)(x)\,dx\\
		&\quad+\int_{-\infty}^{\infty}
		(\px u_{0,n}-\px u_{0})^2\,dx,
	\end{split}
\end{equation*}
where we have used the Lipschitz 
continuity of $\arctan(\cdot)$.
Given that $\px u_{0,n}\to\px u_0$,
$\px v_{0,n}\to\px v_0$, and
$(\px u_{0,n})\px v_{0,n}\to (\px u_{0})\px v_{0}$
in $L^2$, along with the boundedness of $\arctan(\cdot)$, 
and by applying Lemma \ref{lA2}, we 
can conclude that $I_{3,n}\to 0$ 
as $n\to \infty$.

Now let us prove that $I_{4,n}\to 0$ 
as $n\to\infty$. Since $\px u_0\in L^2$, there 
exists a sequence $\{\tilde{f}_m\}_{n\in\mathbb{N}}
\subset C^{\infty}(\mathbb{R})$ with compact support, 
such that $\|\tilde{f}_m-\px u_0\|_{L^2}\to 0$ 
as $m\to\infty$. Using this, we deduce that 
\begin{equation*}
	\begin{split}
		I_{4,n}&\leq 
		2\int_{-\infty}^{\infty}
		\left(
		\arctan(\px u_{0})(y_{0,n}(\xi))-
		\arctan(\tilde{f}_m)(y_{0,n}(\xi))
		\right)^2d\xi\\
		&\quad+
		4\int_{-\infty}^{\infty}
		\left(
		\arctan(\tilde{f}_m)(y_{0,n}(\xi))-
		\arctan(\tilde{f}_m)(y_{0}(\xi))
		\right)^2d\xi\\
		&\quad
		+4\int_{-\infty}^{\infty}
		\left(
		\arctan(\tilde{f}_m)(y_{0}(\xi))-
		\arctan(\px u_{0})(y_{0}(\xi))
		\right)^2d\xi=:
		2I_{4,n,m}^{(1)}+4I_{4,n,m}^{(2)}
		+4I_{4,m}^{(3)}.
	\end{split}
\end{equation*}
By applying similar reasoning as used for $I_{3,n}$, we 
can demonstrate that $I_{4,m}^{(3)} \to 0$ as $m \to \infty$. 
Furthermore, the integral $I_{4,n,m}^{(1)}$ can be 
estimated in the following manner:
\begin{equation*}
	\begin{split}
		I_{4,n,m}^{(1)}&=
		\int_{-\infty}^{\infty}
		\left(\left(\arctan(\tilde{f}_m)-
		\arctan(\px u_{0})\right)^2
		\left(1+(\px u_{0,n})^2\right)
		\left(1+(\px v_{0,n})^2\right)\right)(x)\,dx
		\\
		&\leq
		\int_{-\infty}^{\infty}
		\left(\left(\arctan(\tilde{f}_m)-
		\arctan(\px u_{0})\right)^2
		\left(1+(\px u_{0})^2\right)
		\left(1+(\px v_{0})^2\right)\right)(x)\,dx
		\\
		&\quad+
		\pi^2
		\int_{-\infty}^{\infty}
		\left(
		\left|(\px u_{0,n})^2-(\px u_{0})^2\right|
		+
		\left|(\px v_{0,n})^2-(\px v_{0})^2\right|
		\right.
		\\
		&\qquad\qquad\qquad\left.+
		\left|((\px u_{0,n})\px v_{0,n})^2
		-((\px u_{0})\px v_{0})^2\right|
		\right)(x)\,dx,
	\end{split}
\end{equation*}
To derive this inequality,	
we begin by expanding the product 
$(1 + (\px u_{0,n})^2)\,(1 + (\px v_{0,n})^2)$ 
and comparing it with 
$(1 + (\px u_{0})^2)\,(1 + (\px v_{0})^2)$. 
Their difference turns out to be a sum 
of simpler absolute-value terms 
(for instance, $\bigl|(\px u_{0,n})^2
-(\px u_{0})^2\bigr|$). 
Next, because $\arctan(\cdot)$ takes values 
in $(-\tfrac{\pi}{2}, \tfrac{\pi}{2})$, 
we use the fact that $\bigl(\arctan(\tilde{f}_m)
-\arctan(\px u_{0})\bigr)^2 \le \pi^2$ 
to factor out $\pi^2$ when bounding that difference. 
Putting it all together, we see that the original integral 
can be separated into the main term plus an extra integral 
accounting for these product differences.
This, together with Lemma \ref{lA2}, implies that 
$I_{4,n,m}^{(1)}\to 0$ as $n,m\to\infty$.

Finally, let us show that $I_{4,n,m}^{(2)}\to 0$ 
as $n\to\infty$, for each fixed $m$. Taking into account 
\eqref{y0y0n} and that $\tilde{f}\to\px u_0$ in $L^2$, 
we conclude that for any $\ve>0$ 
there exists $R>0$ such that
\begin{equation}\label{I41}
	\left(\int_{-\infty}^{-R}
	+\int^{\infty}_{R}\right)
	\left(
	\arctan(\tilde{f}_m)(y_{0,n}(\xi))-
	\arctan(\tilde{f}_m)(y_{0}(\xi))
	\right)^2d\xi<\ve,\quad
	\mbox{for all }n,m\in\mathbb{N}.
\end{equation}
Then due to Lipschitz continuity of both $\arctan$ 
and $\tilde{f}_m$, we have the following estimate:
\begin{equation}\label{I42}
	\int_{-R}^{R}
	\left(
	\arctan(\tilde{f}_m)(y_{0,n}(\xi))-
	\arctan(\tilde{f}_m)(y_{0}(\xi))
	\right)^2d\xi
	\leq 2RL_{m,R}^2
	\|y_{0,n}-y_0\|_{L^\infty}^2,
\end{equation}
where $L_{m,R}>0$ is a Lipschitz constant of $f_m$ 
on the interval $[-R,R]$. Combining \eqref{I41} and \eqref{I42} 
we conclude that $I_{4,n,m}^{(2)}\to 0$ 
as $n\to\infty$, for any fixed $m$.

Recalling \eqref{W0Wn}, we derive \eqref{W0W0n} 
for $W_{0,n}$. Similarly, the limit for $Z_{0,n}$ 
can be determined.	

\medskip
	
\textbf{Step 4.} We observe that there is a lack of convergence 
of the initial data in the $L^\infty$ norm. 
Specifically, the limits $\|W_{0,n} - W_0\|_{L^\infty} \to 0$ 
and $\|Z_{0,n} - Z_0\|_{L^\infty} \to 0$ as $n \to \infty$ 
are not known, see \eqref{W0W0n}. Consequently, the method 
outlined in \cite[Lemma 4.1]{HQ21}, which is utilized to 
assess the Lipschitz continuity of the 
right-hand side of \eqref{ODE}, is not applicable 
in this context. Therefore, we 
must adopt an alternative approach. 
Consider (see \cite[Equation (6.8)]{BC07d})
\begin{equation*}
	\begin{split}
		F_n(t):=&\|(U-U_n)(t,\cdot)\|_{L^\infty}
		+\|(V-V_n)(t,\cdot)\|_{L^\infty}
		+\|(W-W_n)(t,\cdot)\|_{L^2}\\
		&+\|(Z-Z_n)(t,\cdot)\|_{L^2}
		+\|(q-q_n)(t,\cdot)\|_{L^2},
	\end{split}
\end{equation*}
where $(U,V,W,Z,q)$ and $(U_n,V_n,W_n,Z_n,q_n)$ 
are global solutions of the ODE 
system \eqref{ODE} with initial data 
$(U_0,V_0,W_0,Z_0,q_0)$ and 
$(U_{0,n},V_{0,n},W_{0,n},Z_{0,n},q_{0,n})$, 
respectively, see Theorem \ref{gwp}.

Our goal is to prove that
\begin{equation}\label{Fd}
	\frac{d}{dt}F_n(t)\leq C_T F_n(t),\quad 
	t\in [0,T],
\end{equation}
with some $C_T>0$ and any $T>0$. Recalling the system 
\eqref{ODE} and taking into account the uniform 
bounds derived in Section \ref{GS}, to prove 
\eqref{Fd}, it is enough to show that
\begin{equation}\label{PjPjn}
	\begin{split}
		&\|(P_j-P_{j,n})(t,\cdot)\|_{L^2\cap L^\infty},
		\,\|(\px P_j-\px P_{j,n})(t,\cdot)\|_{L^2\cap L^\infty}
		\leq C_T F_n(t), \quad j=1,2,\\
		&\|(S_j-S_{j,n})(t,\cdot)\|_{L^2\cap L^\infty},
		\,\|(\px S_j-\px S_{j,n})(t,\cdot)\|_{L^2\cap L^\infty}
		\leq C_T F_n(t), \quad j=1,2,
	\end{split}
\end{equation}
where $P_{j,n}$, $\px P_{j,n}$, $S_{j,n}$, 
and $\px S_{j,n}$ are defined analogously to 
$P_{j}$, $\px P_{j}$, $S_{j}$, and $\px S_{j}$ 
as described in Proposition \ref{Pj}. 
However, in this case, the variables 
$(U_n,V_n,W_n,Z_n,q_n)$ replace $(U,V,W,Z,q)$.

Let $E_{u_{0,n}}$, $E_{v_{0,n}}$, 
and $H_{0,n}$ be as in \eqref{consq}, but with $u_{0,n}$ 
and $v_{0,n}$ instead of $u_0$ and $v_0$, respectively. 
Since $u_{0,n}\to u_0$, $v_{0,n}\to v_0$ in $H^1$ 
and $(\partial_x u_{0,n})\partial_x v_{0,n}\to 
(\partial_x u_{0})\partial_x v_{0}$ in $L^2$, we have 
$E_{u_{0,n}}\to E_{u_{0}}$, $E_{v_{0,n}}\to E_{v_{0}}$, 
and $H_{0,n}\to H_0$. This implies that the a priori 
estimates involving $E_{u_{0,n}}$, $E_{v_{0,n}}$, 
and $H_{0,n}$ can be expressed in terms of 
$E_{u_{0}}$, $E_{v_{0}}$, $H_{0}$ and some radius $R>0$, 
such that
$$
|E_{u_{0,n}}-E_{u_{0}}|,\,
|E_{v_{0,n}}-E_{v_{0}}|,\,
|H_{0,n}-H_0|\leq R,\quad
\mbox{for all }n\in\mathbb{N}.
$$
Therefore, for simplicity, we will use the 
same notations for the a priori bounds, regardless 
of whether they involve $n$ or not.
	
Consider (see \eqref{pxP2-L} and \eqref{I1I2})
\begin{equation*}
	I_5=\int_{\xi}^{\infty}
	\left|
	\mathcal{E}(\xi,\eta)
	\left(q\sin^2\frac{W}{2}\sin Z\right)(\eta)
	-\mathcal{E}_n(\xi,\eta)
	\left(q_n\sin^2\frac{W_n}{2}\sin Z_n
	\right)(\eta)\right|\,d\eta.
\end{equation*}
Then we have (see \eqref{I1e} and \eqref{qap})
\begin{equation*}
	\begin{split}
		I_5\leq &
		\int_{\xi}^{\infty}
		\tilde{\Gamma}(\xi-\eta)
		\left|
		\left(q\sin^2\frac{W}{2}\sin Z\right)(\eta)
		-\left(q_n\sin^2\frac{W_n}{2}\sin Z_n\right)(\eta)
		\right|\,d\eta\\
		&
		+\kappa_0\int_{\xi}^{\infty}
		\left|
		\mathcal{E}(\xi,\eta)-\mathcal{E}_n(\xi,\eta)
		\right|
		|Z(\eta)|\,d\eta
		=: 
		I_{5,1}+\kappa_0 I_{5,2}.
	\end{split}
\end{equation*}
The integral $I_{5,1}$ can be 
estimated as follows (see \eqref{I11-e0}):
\begin{equation*}
	\begin{split}
		I_{5,1}\leq \tilde{\Gamma}\ast|q-q_n|
		+\kappa_0\left(
		\tilde{\Gamma}\ast|W-W_n|
		+\tilde{\Gamma}\ast|Z-Z_n|
		\right).
	\end{split}
\end{equation*}
Using the Cauchy-Schwarz inequality, we obtain
\begin{equation*}
	\begin{split}
		\left\|I_{5,1}\right\|_{L^\infty}
		\leq\left\|\tilde{\Gamma}\right\|_{L^2}
		\left\|q-q_n\right\|_{L^2}
		+\kappa_0\left\|\tilde{\Gamma}\right\|_{L^2}
		\left(\left\|W-W_n\right\|_{L^2}
		+\left\|Z-Z_n\right\|_{L^2}
		\right),
	\end{split}
\end{equation*}
while applying \eqref{coest} 
with $p=2$ yields
\begin{equation*}
	\left\|I_{5,1}\right\|_{L^2}
	\leq\left\|\tilde{\Gamma}\right\|_{L^1}
	\left\|q-q_n\right\|_{L^2}
	+\kappa_0\left\|\tilde{\Gamma}\right\|_{L^1}
	\left(\left\|W-W_n\right\|_{L^2}
	+\left\|Z-Z_n\right\|_{L^2}\right).
\end{equation*}
	
We have the following estimate 
for $I_{5,2}$ (cf.\,\,\eqref{I12-e}):
\begin{equation*}
	\begin{split}
		I_{5,2}&\leq
		\int_{\xi}^{\infty}
		\tilde{\Gamma}(\xi-\eta)|\xi-\eta|
		|Z(\eta)||(q-q_n)(\eta)|\,d\eta\\
		&\quad
		+\int_{\xi}^{\infty}
		\tilde{\Gamma}(\xi-\eta)|Z(\eta)|
		\int_\xi^\eta\left(
		|W-W_n|+|Z-Z_n|
		\right)(s)\,ds\,d\eta\\
		&\leq
		\|Z\|_{L^\infty}
		\left(|\xi|\tilde{\Gamma}(\xi)\right)
		\ast|q-q_n|
		+(\|W-W_n\|_{L^2}+\|Z-Z_n\|_{L^2})
		\left(|\xi|^{1/2}\tilde{\Gamma}(\xi)\right)
		\ast|Z|,
	\end{split}
\end{equation*}
where we have used 
the Cauchy-Schwarz inequality for 
$\int_\xi^\eta(\cdot)\,ds$.
Using again the Cauchy-Schwarz inequality, we obtain
\begin{equation*}
	\left\|I_{5,2}\right\|_{L^\infty}
	\leq\|Z\|_{L^\infty}
	\left(
	\|\xi\tilde{\Gamma}(\xi)\|_{L^2}
	\|q-q_n\|_{L^2}
	+\||\xi|^{1/2}\tilde{\Gamma}(\xi)\|_{L^1}
	(\|W-W_n\|_{L^2}+\|Z-Z_n\|_{L^2})
	\right),
\end{equation*}
while using \eqref{coest} with $p=2$, 
we arrive at
\begin{equation*}
	\left\|I_{5,2}\right\|_{L^2}
	\leq\|Z\|_{L^\infty}
	\|\xi\tilde{\Gamma}(\xi)\|_{L^1}
	\|q-q_n\|_{L^2}
	+\|Z\|_{L^2}
	\||\xi|^{1/2}\tilde{\Gamma}(\xi)\|_{L^1}
	(\|W-W_n\|_{L^2}+\|Z-Z_n\|_{L^2}).
\end{equation*}
By the uniform bounds \eqref{WZe} 
for $\|Z\|_{L^p}$, $p=2,\infty$,
we achieve the required estimate for $\|I_5\|_{L^2\cap L^\infty}$.

By estimating the integrals involved in 
$(P_j-P_{j,n})$, $(\px P_j-\px P_{j,n})$, $(S_j-S_{j,n})$ 
and $(\px S_j-\px S_{j,n})$ in a similar manner to how 
we estimated $I_5$ above, we can 
derive \eqref{PjPjn}, thereby establishing \eqref{Fd}. 
Moreover, we can also demonstrate that
\begin{equation}\label{Fd-}
	\frac{d}{dt}F_n(-t)\leq C_TF_n(-t),
	\quad t\in[0,T],
\end{equation}
with some $C_T>0$ and any $T>0$.

\medskip

\textbf{Step 5.} By applying Gronwall’s inequality to 
\eqref{Fd} and \eqref{Fd-}, and using 
\eqref{U0U0n} and \eqref{W0W0n}, along with 
the condition $q_0=q_{0,n}=1$, we can conclude that
\begin{equation*}
	\sup\limits_{t\in[-T,T]}F_n(t)\to 0,
	\quad n\to \infty,
\end{equation*}
for all $T>0$. In particular, this implies that
\begin{equation}\label{UVUnVn}
	\sup\limits_{t\in[-T,T]}
	\left(
	\|(U-U_n)(t,\cdot)\|_{L^\infty}
	+\|(V-V_n)(t,\cdot)\|_{L^\infty}
	\right)\to 0,\quad
	n\to\infty.
\end{equation}
From \eqref{char1}, \eqref{y0y0n}, 
\eqref{UVUnVn}, and \eqref{apest1}, we 
can deduce that
\begin{equation}\label{yyn}
	\sup\limits_{t\in[-T,T]}
	\|(y-y_n)(t,\cdot)\|_{L^\infty}
	\to 0,\quad n\to\infty,
\end{equation}
where $y_n$ is defined by \eqref{char1} with
$y_{0,n}$, $U_n$ and $V_n$ instead of 
$y_0$, $U$ and $V$, respectively.
Combining \eqref{uvdef}, \eqref{UVUnVn}, \eqref{yyn}, 
and \eqref{uH}, we obtain \eqref{dsm}.
\end{proof}

Combining Propositions \ref{Pgws}, \ref{wsol}, 
and \ref{Pdsm}, we obtain the following global 
existence theorem:

\begin{theorem}[Global conservative 
weak solution]\label{gws}
Assume that $(u_0,v_0)\in\Sigma$, where
$(\Sigma,d_\Sigma)$ is defined 
by \eqref{Sig} and \eqref{d^2}.
Then there exists a global conservative weak 
solution $(u,v)$ of the Cauchy 
problem \eqref{t-c-N-n}-\eqref{iid}, 
defined by \eqref{uvdef}, in the sense 
of Definitions \ref{defsc}. Moreover, this 
solution $(u,v)$ satisfies the continuity 
property of the data-to-solution map 
given in Proposition \ref{Pdsm}.
\end{theorem}

\section{A semigroup of global solutions}\label{sgs}

The global solutions of the two-component system 
constructed in Theorem \ref{gws} 
do not constitute a semigroup.
Indeed, in the particular case of the Novikov equation, 
the energy concentration of $(\px u)^4\,dx$ occurs, when 
two peakons collide \cite[Section 8]{CCCS18}. 
Therefore, in addition to a solution $(u,v)\in\Sigma$, 
we consider a positive Radon measure $\mu$ on $\mathbb{R}$ 
which satisfies \eqref{mu}. Then for the initial data 
$(u_0,v_0,\mu_0)\in\mathcal{D}$, 
see \eqref{D}, we define $y_0(\xi)$ as follows:
\begin{equation}\label{y0m}
	y_0(\xi)=
	\begin{cases}
		\sup\left\{
		y\in\mathbb{R}:
		y+\mu_0\left([0,y]\right)
		\leq\xi\right\},&\xi\geq 0,\\
		\inf\left\{
		y\in\mathbb{R}:
		|y|+\mu_0\left([-y,0)\right)
		\leq-\xi\right\},&\xi< 0.
	\end{cases}
\end{equation}
In the case $\mu=\mu^{ac}$, the above 
definition is equivalent to \eqref{y0}.
Moreover, taking into account that 
$$
\pxi y_0(\xi)\leq
\frac{1}{\left(1+(\px u_0)^2\right)
\left(1+(\px v_0)^2\right)(y_0(\xi))}
\quad\mbox{for a.e. }\xi\in\mathbb{R},
$$
we conclude that $U_0, V_0 \in H^1 \cap W^{1,4}$, 
see \eqref{U-L4}. Consequently, we can 
examine the Cauchy problem for the 
semilinear system represented by \eqref{ODE}-\eqref{id} 
within the space $E$ defined in \eqref{EB}, following 
the approach outlined in Section \ref{Gsl}. 
Define the map
\begin{equation}\label{Psit}
	\Psi_t(u_0,v_0,\mu_0)
	=\left(u(t),v(t),\mu_{(t)}\right),\quad t\in\mathbb{R},
\end{equation}
where $(u,v)(t)$ is given by \eqref{uvdef},
while $\mu_{(t)}$ is defined in terms 
of the global solution $(U,V,W,Z,q)$ of the 
ODE system \eqref{ODE}-\eqref{id} as follows:
\begin{equation}\label{mut}
	\mu_{(t)}([a,b])
	=\int\limits_{\{\xi:y(t,\xi)\in[a,b]\}}
	\left(
	\cos^2\frac{W}{2}\sin^2\frac{Z}{2}
	+\sin^2\frac{W}{2}\cos^2\frac{Z}{2}
	+\sin^2\frac{W}{2}\sin^2\frac{Z}{2}
	\right)(t,\xi)q(t,\xi)\,d\xi,
\end{equation}
with $y(t,\xi)$ defined in \eqref{char1}.

Observe that if $\partial_\xi y(t,\xi) \neq 0$ for 
a.e.~$\xi$, such that $y(t,\xi)$ lies within 
the interval $[a,b]$, it is possible to perform 
a change of variables in equation 
\eqref{mut} by setting $x = y(t,\xi)$:
\begin{equation*}
	\mu_{(t)}([a,b])
	=\int_{a}^{b}
	\left(
	(\px u)^2+(\px v)^2+(\px u)^2(\px v)^2
	\right)(t,x)\,dx,
\end{equation*}
and therefore
\begin{equation*}
	d\mu^{ac}_{(t)}
	=\left(
	(\px u)^2+(\px v)^2+(\px u)^2(\px v)^2
	\right)dx.
\end{equation*}

We are now ready to 
prove Theorem \ref{Thm}.

\begin{proof}[Proof of Theorem \ref{Thm}]
\textbf{Item (1)} follows from Theorem \ref{gws}.

\medskip

Let us prove \textbf{item (2)}. 
The fact that $(u_0,v_0)=(u(0),v(0))$ follows directly 
from \eqref{id} and \eqref{uvdef}. 
In order to show that $\mu_{(0)}=\mu_0$, we observe 
that \eqref{y0m} implies (see \cite[Page 236]{BC07})
\begin{equation}\label{mu0B}
	\mu_0([a,b])+\mathrm{meas}([a,b])
	=\mathrm{meas}\left\{
	\xi\in\mathbb{R}:y_0(\xi)\in[a,b]
	\right\}.
\end{equation}
Indeed, consider $0\leq a\leq b$, where 
$y_0(\xi)\in[a,b]$ for $\xi\geq0$ only.
Introducing the function $f(y)=y+\mu_0([0,y])$, 
we have from \eqref{y0m} that
$y_0(\xi)=\sup\{y\in\mathbb{R}
:f(y)\leq\xi\}$, $\xi\geq 0$.
Taking into account that $f(y)$ 
is strictly monotone, we have
$$
\{\xi\in\mathbb{R}:y_0(\xi)\in[a,b]\}
=[f(a-),f(b+)].
$$
Observing that 
$f(b+)-f(a-)=b-a+\mu_0([0,b])-\mu_0([0,a))$, 
we conclude \eqref{mu0B} for $0\leq a\leq b$.
The other values of $a$ and $b$ 
can be treated similarly.
		
By definition of the singular part of 
the Radon measure $\mu_0$ with respect 
to the Lebesgue measure (see, 
e.g., \cite[Section 1.6.2]{EG92}), 
there exists a Borel set 
$\mathcal{B}\subset\mathbb{R}$, such 
that $\mathrm{meas}(\mathcal{B})=
\mu_0^s(\mathbb{R}\setminus\mathcal{B})=0$. 
Taking into account that \eqref{mu0B} is 
valid for any Borel set, we have
\begin{equation*}
	\mu_0^s(\mathcal{B})
	=\mathrm{meas}\left\{
	\xi\in\mathbb{R}:y_0(\xi)\in\mathcal{B}
	\right\},\quad
	\mathrm{meas}(\mathcal{B})=0.
\end{equation*}
Therefore, recalling that $y_0(\xi)$ 
is monotone increasing, we obtain
\begin{equation}\label{mu0s}
	\begin{split}
		&\mu_0^{s}([a,b])=
		\mathrm{meas}\left\{
		\xi\in\mathbb{R}:y_0(\xi)\in[a,b]
		\mbox{ and }(\pxi y_0)(\xi)=0\right\}.
	\end{split}
\end{equation}
From equations \eqref{mut} and \eqref{mu0s}, it 
follows that (considering $\xi \in \mathbb{R}$ such 
that $y_0(\xi) \in [a,b]$ for all integrals, and noting 
that $\pxi y_0(\xi) \neq 0$ if and only if 
$\cos\frac{W(\xi)}{2} \neq 0$ 
and $\cos\frac{Z(\xi)}{2} \neq 0$):

\begin{equation*}
\begin{split}
	\mu_{(0)}([a,b])&=\left(
	\int_{\{\pxi y_0(\xi)\neq0\}}
	+\int_{\left\{\cos\frac{W_0(\xi)}{2}=0,\,
	\cos\frac{Z_0(\xi)}{2}\neq0\right\}}
	+\int_{\left\{\cos\frac{W_0(\xi)}{2}\neq0,\,
	\cos\frac{Z_0(\xi)}{2}=0\right\}}\right.\\
	&\left.
	\qquad+\int_{\left\{\cos\frac{W_0(\xi)}{2}=0,\,
	\cos\frac{Z_0(\xi)}{2}=0\right\}}
	\right)(\cdot)\,d\xi\\
	&=\mu_0^{ac}([a,b])+\mathrm{meas}\left\{
	\xi\in\mathbb{R}:y_0(\xi)\in[a,b]
	\mbox{ and }(\pxi y_0)(\xi)=0\right\}
	=\left(\mu_0^{ac}+\mu_0^{s}\right)([a,b]).
\end{split}
\end{equation*}
Thus, we have proved that 
$\Psi_0=\mathrm{id}$, see \eqref{Psit}.
	
Let us show that 
$\Psi_{t+\tau}=\Psi_{t}\circ\Psi_{\tau}$.
Using the notation 
$$
(\tilde{u}_0,\tilde{v}_0,\tilde{\mu}_0)
=\Psi_\tau(u_0,v_0,\mu_0),
$$
our goal is to prove that
$$
\Psi_{t+\tau}(u_0,v_0,\mu_0)
=\Psi_{t}(\tilde{u}_0,\tilde{v}_0,\tilde{\mu}_0).
$$
Recalling \eqref{pxiy} we deduce that 
(see \cite[equation (8.5)]{BC07d})
\begin{equation*}
	\begin{split}
		&y(\tau,\xi)
		+\int_{\xi_0}^{\xi}
		\left(
		\cos^2\frac{W}{2}\sin^2\frac{Z}{2}
		+\sin^2\frac{W}{2}\cos^2\frac{Z}{2}
		+\sin^2\frac{W}{2}\sin^2\frac{Z}{2}
		\right)(\tau,\eta)q(\tau,\eta)\,d\eta
		=\sigma(\xi),\\
		&\sigma(\xi)=\int_{\xi_0}^{\xi}q(\tau,\eta)\,d\eta,
	\end{split}
\end{equation*}
where we have taken the smallest $\xi_0=\xi_0(\tau)$ 
such that $y(\tau,\xi_0)=0$.
	
Define $\tilde{y}_0$ 
by the right-hand side of \eqref{y0m} 
with $\tilde\mu_0$ instead of $\mu_0$, where
$\tilde\mu_0$ is given by \eqref{mut} with $t=\tau$.
Let us prove that 
$\tilde{y}_0(\sigma(\xi))=y(\tau,\xi)$.
Consider $\xi>\xi_0$. Taking into account that
$$
\int_{\xi_0}^{\xi}
\left(\cos^2\frac{W}{2}\sin^2\frac{Z}{2}
+\sin^2\frac{W}{2}\cos^2\frac{Z}{2}
+\sin^2\frac{W}{2}\sin^2\frac{Z}{2}
\right)(\tau,\eta)q(\tau,\eta)\,d\eta
\leq\tilde{\mu}_0([0,y(\tau,\xi)]),
$$
we conclude that 
$\tilde{y}_0(\sigma(\xi))\leq y(\tau,\xi)$. On the other hand, 
since $\xi_0$ is the smallest real number such 
that $y(\tau,\xi_0)=0$, we have 
for any small $\ve>0$
$$
\int_{\xi_0}^{\xi}
\left(\cos^2\frac{W}{2}\sin^2\frac{Z}{2}
+\sin^2\frac{W}{2}\cos^2\frac{Z}{2}
+\sin^2\frac{W}{2}\sin^2\frac{Z}{2}
\right)(\tau,\eta)q(\tau,\eta)\,d\eta
\geq\tilde{\mu}_0([0,y(\tau,\xi)-\ve]),
$$
which implies that
$\tilde{y}_0(\sigma)\geq y(\tau,\xi)-\ve$. 
Employing a similar line of reasoning for $\xi < 0$, we 
deduce that $\tilde{y}_0(\sigma(\xi)) 
= y(\tau, \xi)$ for all 
$\xi \in \mathbb{R}$.
	
Using that $\tilde{y}_0(\sigma)=y(\tau,\xi)$, 
we obtain
$$
\tilde{U}_0(\sigma)
:=\tilde{u}_0(\tilde{y}_0(\sigma))
=U(\tau,[y(\tau)]^{-1}(\tilde{y}_0(\sigma)))
=U(\tau,\xi).
$$
Arguing similarly, we arrive 
at the following expressions:
$$
\tilde{V}_0(\sigma)=V(\tau,\xi),\quad
\tilde{W}_0(\sigma)=W(\tau,\xi),\quad
\tilde{Z}_0(\sigma)=(\tau,\xi),
$$
where $\tilde{V}_0,\tilde{W}_0,\tilde{Z}_0$ are 
defined in terms of $\tilde{u}$ 
and $\tilde{y}_0$ by \eqref{id}. 
Finally, observing that the differential equation 
for $\pt q$ in \eqref{ODE} is linear with respect 
to $q$ and taking into account the 
definitions of $P_j$, $S_j$, $\px P_j$, 
and $\px S_j$, $j=1,2$, see Proposition \ref{Pj}, 
we conclude that the flows
$$
\left(\tilde{U}(t,\sigma),\tilde{V}(t,\sigma),
\tilde{W}(t,\sigma),\tilde{Z}(t,\sigma),
\tilde{q}(t,\sigma)q(\tau,\xi)\right)
\quad
\mbox{and}
\quad
(U,V,W,Z,q)(t+\tau,\xi),
$$
satisfy the same differential equation \eqref{ODE} 
with the same initial data at $t=0$.
In view of the uniqueness of the solution of 
the ODE system, these flows are identical 
and thus $\Psi_{t+\tau}=\Psi_{t}\circ\Psi_{\tau}$.

\medskip
	
Let us prove \textbf{item (3)}. We will now demonstrate 
that $y_{0,n} \to y_0$ in $L^\infty(\mathbb{R})$, 
comparing with Step 1 of Proposition \ref{Pdsm}. 
Here, $y_{0,n}$ is defined 
similarly to $y_0$, but with $\mu_{0,n}$ 
replacing $\mu_0$. 
Consider any fixed $\xi \geq 0$.
Then, by Definition \ref{y0m} 
of $y_0(\xi)$ with $\xi\geq 0$, 
for any $\ve>0$ there exists 
$\tilde{y}\in\mathbb{R}$ such 
that $|y_0(\xi)-\tilde{y}|<\ve$ and
\begin{equation}\label{in1}
	\xi-\ve\leq\tilde{y}
	+\mu_0([0,\tilde{y}])\leq\xi.
\end{equation}
Using that $\mu_{0,n}([0,\tilde{y}])
\to\mu_{0}([0,\tilde{y}])$ 
as $n\to\infty$ (see, e.g., \cite[Section 1.9, 
Theorem 1]{EG92}), there 
exists $N\in\mathbb{N}$ such 
that for all $n\geq N$ the 
following inequalities hold:
\begin{equation}\label{in2}
	-\ve\leq\mu_{0,n}([0,\tilde{y}])
	-\mu_0([0,\tilde{y}])\leq\ve.
\end{equation}
Combing \eqref{in1} and \eqref{in2} we obtain
$\xi-2\ve\leq\tilde{y}+\mu_{0,n}([0,\tilde{y}])\leq\xi+\ve$, 
which implies that $|y_{0,n}(\xi)-\tilde{y}|\leq2\ve$ 
for all $n\geq N$. 
Recalling that $|y_0(\xi) - \tilde{y}| < \ve$, 
we deduce that $|y_{0,n}(\xi) - y_0(\xi)| \leq 3\varepsilon$. 
Using a similar argument for $\xi < 0$, we conclude 
that $y_{0,n} \to y_0$ pointwise for 
all $\xi \in \mathbb{R}$.
	
Given that (i) $y_{0,n}(\xi)$ is continuous 
and monotonically increasing for all $n$, and (ii) $y_0(\xi)$ 
is continuous, pointwise convergence implies 
uniform convergence on every compact interval $[-R,R]$, 
where $R>0$. Finally, using that 
$\mu_0(\mathbb{R})<\infty$, we arrive at
$$
\|y_{0,n}-y_0\|_{L^\infty(\mathbb{R})}\to 0,
\quad\mbox{as }n\to\infty,
$$
which is nothing but 
\eqref{y0y0n}, see the proof of Proposition \ref{Pdsm} (Step 1).

Based on Steps 2 and 3 of the proof of 
Proposition \ref{Pdsm}, the 
inequalities \eqref{U0U0n} and \eqref{W0W0n} are derived 
from \eqref{y0y0n}, along with the convergences 
$u_{0,n} \to u_0$, $v_{0,n} \to v_0$ in $H^1$, 
and $(\px u_{0,n}) \px v_{0,n} \to (\px u_0) \px v_0$ in $L^2$. 
These convergences correspond to condition a) of 
item (3) in Theorem \ref{Thm}, thereby 
establishing that
\begin{equation*}
	\left\|
	U_{0,n}-U_0
	\right\|_{L^\infty}
	\to 0,\,\,
	\left\|
	V_{0,n}-V_0
	\right\|_{L^\infty}
	\to 0,\quad n\to\infty,
\end{equation*}
and
\begin{equation*}
	\|W_{0,n}-W_0\|_{L^2},\,\,
	\|Z_{0,n}-Z_0\|_{L^2}\to 0,
	\quad n\to\infty.
\end{equation*}
By applying the same reasoning as in Step 4 of 
the proof of Proposition \ref{Pdsm}, we derive 
the inequalities \eqref{Fd} and \eqref{Fd-}. These inequalities 
lead directly to \eqref{unifc1}, as in Step 5 
of the proof of Proposition \ref{Pdsm}.

\medskip
	
\textbf{Item (4)} follows from \eqref{u_x^2n} and \eqref{cn}.

\medskip
	
Finally, we establish \textbf{item (5)}. For any $\Phi(t,\xi) = \phi(t,x)$ 
where $x = y(t,\xi)$, it follows that (the arguments $t,\xi$ 
are omitted here; see the proof of item (4) 
in Proposition \ref{mcl})
\begin{equation*}
	\begin{split}
		I_6&:=
		\int_{-T}^T\int_{-\infty}^{\infty}
		\left(
		\cos^2\frac{W}{2}\sin^2\frac{Z}{2}
		+\sin^2\frac{W}{2}\cos^2\frac{Z}{2}
		+\sin^2\frac{W}{2}\sin^2\frac{Z}{2}
		\right)q
		\pt\Phi\,d\xi\,dt\\
		&=
		\int_{-T}^T\int_{-\infty}^{\infty}
		(\pt\phi+uv\px\phi)(t,x)\,d\mu_{(t)}\,dt,
	\end{split}
\end{equation*}
where $x=y(t,\xi)$. Conversely, by integrating 
by parts and utilizing \eqref{cn-t2}, \eqref{cn-t3}, 
and the identities
\begin{equation*}
	\pt\left(
	q\cos^2\frac{W}{2}\sin^2\frac{Z}{2}
	\right)=
	q\cos^2\frac{W}{2}\sin Z\left(
	UV^2-S_1-\px S_2
	\right)
	+\frac{q}{2}V\sin W\sin^2\frac{Z}{2},
\end{equation*}
we obtain the following expression 
for $I_6$ (here and below we drop the 
argument $t$ in $D_W$ and $D_Z$):
\begin{equation}\label{chI}
	\begin{split}
		I_6&=
		-\int_{-T}^T
		\int_{\mathbb{R}\setminus\{D_W\cup D_Z\}}
		\tilde{A}_2\Phi\,d\xi\,dt
		-\int_{-T}^T\int_{D_Z}
		\left(
		q\sin W\left(\frac{V}{2}
		-P_1-\px P_2\right)	
		\right)\Phi\,d\xi\,dt\\
		&\quad\,\,
		-\int_{-T}^T\int_{D_W}
		\left(
		q\sin Z\left(\frac{U}{2}
		-S_1-\px S_2\right)	
		\right)\Phi\,d\xi\,dt,
	\end{split}
\end{equation}
with
\begin{equation*}
	\begin{split}
		\tilde{A}_2=&\,qUV(U\sin W+V\sin Z)
		+q\sin W\left(\frac{V}{2}\sin^2\frac{Z}{2}
		-P_1-\px P_2\right)\\
		&+q\sin Z\left(\frac{U}{2}\sin^2\frac{W}{2}
		-S_1-\px S_2\right).
	\end{split}
\end{equation*}
Given that $\mathrm{meas}(N_W)+\mathrm{meas}(N_Z)=0$, 
by substituting $x=y(t,\xi)$ into \eqref{chI} and 
recalling \eqref{pxi-y}, we conclude that $\mu_{(t)}$ 
serves as a measure-valued solution 
to the desired transport equation with 
source term \eqref{mvm1} (see also Remark \ref{wms}).
\end{proof}	

\appendix
\numberwithin{equation}{subsection}
\numberwithin{lemma}{subsection}
\numberwithin{proposition}{subsection}

\section*{Appendix}
\renewcommand{\thesubsection}{\Alph{subsection}}

In this appendix, we provide 
detailed proofs of specific 
propositions and technical lemmas utilized elsewhere 
in the article. Although many of these proofs are available 
in the existing literature, certain elements of the analysis are 
applied in establishing the new results presented in this paper.

\subsection{Appendix A}\label{AS4}

We will present a 
summary of the proofs of  Lemma \ref{L1} 
and Propositions \ref{POm}, \ref{LipP}.

\begin{proof}[Proof of Lemma \ref{L1}]
For any function $W(t,\cdot)$ that 
satisfies the conditions in \eqref{Om}, 
we have the following:
\begin{equation}\label{mW}
	\begin{split}
		\mathrm{meas}\left\{
		\xi\in\mathbb{R}:\frac{\pi}{4}\leq
		\left|\frac{W(\xi)}{2}\right|\leq\frac{3\pi}{4}
		\right\}&\leq
		\mathrm{meas}\left\{
		\xi\in\mathbb{R}:
		\left|\sin\frac{W(\xi)}{2}\right|\geq\frac{1}{\sqrt{2}}
		\right\}\\
		&=\int_{\left\{
		\xi\in\mathbb{R}\,:\,
		\sin^2\frac{W(\xi)}{2}\geq\frac{1}{2}
		\right\}}\,ds\\
		&\leq 2\int_{\left\{
		\xi\in\mathbb{R}\,:\,
		\sin^2\frac{W(\xi)}{2}\geq\frac{1}{2}
		\right\}}
		\sin^2\frac{W(s)}{2}\,ds\\
		&\leq \frac{1}{2}\int_{-\infty}^{\infty}W^2(s)\,ds
		\leq \frac{R_2^2}{2}.
	\end{split}
\end{equation}
Arguing similarly for $Z$, we conclude that
\begin{equation}\label{mZ}
	\mathrm{meas}\left\{
	\xi\in\mathbb{R}:\frac{\pi}{4}\leq
	\left|\frac{Z(\xi)}{2}\right|\leq\frac{3\pi}{4}
	\right\}\leq \frac{R_2^2}{2}.
\end{equation}
Introducing the sets
\begin{equation}\label{D1}
	D_1=\left\{s\in[\eta,\xi]:
	\left|\frac{W(s)}{2}\right|\leq\frac{\pi}{4}\right\},\quad
	D_2=\left\{s\in[\eta,\xi]:
	\left|\frac{Z(s)}{2}\right|\leq\frac{\pi}{4}\right\},\quad
	\eta\leq\xi,
\end{equation}
and combining \eqref{mW} and \eqref{mZ}, 
we obtain the following estimate (see \eqref{Om}):
\begin{equation}\label{Ep}
	\begin{split}
		\int_\eta^\xi q(s)\cos^2\frac{W(s)}{2}
		\cos^2\frac{Z(s)}{2}\,ds
		&\geq \int_{D_1\cap D_2} q(s)\cos^2\frac{W(s)}{2}
		\cos^2\frac{Z(s)}{2}\,ds
		\geq \frac{1}{4}\int_{D_1\cap D_2}q(s)\,ds\\
		&\geq\frac{q^-}{4}\int_{D_1\cap D_2}\,ds
		\geq\frac{q^-}{4}\left(
		\int_\eta^\xi\,ds-\int_{[\eta,\xi]\setminus D_1}\,ds
		-\int_{[\eta,\xi]\setminus D_2}\,ds\right)\\
		&\geq\frac{q^-}{4}\left(\xi-\eta-R_2^2\right),
	\end{split}
\end{equation}
where $\eta\leq\xi$.
Thus for all $\xi,\eta\in \mathbb{R}$ we have
\begin{equation*}
	\left|\int_\eta^\xi q(s)\cos^2\frac{W(s)}{2}
	\cos^2\frac{Z(s)}{2}\,ds\right|
	\geq \frac{q^-}{4}\left(|\xi-\eta|-R_2^2\right),
\end{equation*}
which implies \eqref{E-est}.
\end{proof}

\begin{proof}[Proof of Proposition \ref{POm}]
We present a detailed proof demonstrating 
that $\px P_2 \in H^1 \cap W^{1,4}$. The analysis 
for the other nonlocal terms follows similarly. 
For clarity and simplicity, we omit 
the argument $t$ in all computations.

Considering that $\sin^2\frac{W(\eta)}{2}\leq 1$, 
$|Z(\eta)|\leq |\sin Z(\eta)|$, and
using \eqref{E-est}, we can derive 
from \eqref{pxP_2} that
\begin{equation*}
	|\px P_2(\xi)|\leq\frac{q^+}{8}
	\int_{-\infty}^{\infty}
	\Gamma(\xi-\eta)|Z(\eta)|\,d\eta,
\end{equation*}
which implies that (see \eqref{WZ-L4} and \eqref{G-e})
\begin{equation*}
	\begin{split}
		\|\px P_2\|_{L^2\cap L^4}\leq
		\frac{q^+}{8}\|\Gamma\|_{L^1}\|Z\|_{L^2\cap L^4}
		\leq \frac{q^+}{q^-}
		\exp\left(\frac{q^-}{4}R_2^2\right)
		\left(R_2+\left(\frac{3\pi}{2}R_2\right)^{1/2}\right).
	\end{split}
\end{equation*}
Differentiating \eqref{pxP_2} with respect to $\xi$ we arrive at
\begin{equation}\label{xipxP_2}
	\begin{split}
		\pxi(\px P_2)=
		-\frac{q(\xi)}{4}\sin^2\frac{W(\xi)}{2}\sin Z(\xi)
		+\frac{1}{8}\int_{-\infty}^{\infty}
		&E(t,\xi,\eta)q^2(\eta)\cos^2\frac{W(\eta)}{2}
		\cos^2\frac{Z(\eta)}{2}\\
		&\times\sin^2\frac{W(\eta)}{2}\sin Z(\eta)\,d\eta.
	\end{split}
\end{equation}
Using that $\sin^2\frac{W(\xi)}{2}\leq 1$, 
$\cos^2\frac{W(\eta)}{2}\leq 1$
and $|\sin Z(\xi)|\leq|Z(\xi)|$, we deduce 
from \eqref{xipxP_2} and \eqref{E-est} that
\begin{equation*}
	|\pxi(\px P_2(\xi))|\leq\frac{q^+}{4}|Z(\xi)|+
	\frac{(q^+)^2}{8}
	\int_{-\infty}^{\infty}
	\Gamma(\xi-\eta)|Z(\eta)|\,d\eta,
\end{equation*}
and therefore (see \eqref{WZ-L4} and \eqref{G-e})
\begin{equation*}
	\begin{split}
		\|\pxi(\px P_2)\|_{L^2\cap L^4}&\leq
		\frac{q^+}{4}
		\|Z\|_{L^2\cap L^4}
		+\frac{(q^+)^2}{8}\|\Gamma\|_{L^1}\|Z\|_{L^2\cap L^4}\\
		&\leq q^+
		\left(R_2+\left(\frac{3\pi}{2}R_2\right)^{1/2}\right)
		\left(\frac{1}{4}
		+\frac{q^+}{q^-}
		\exp\left(\frac{q^-}{4}R_2^2\right)
		\right).
	\end{split}
\end{equation*}
\end{proof}

\begin{proof}[Proof of Proposition \ref{LipP}]
We prove \eqref{Pj-L} for 
$\|\px P_{j1}-\px P_{j2}\|_{H^1\cap W^{1,4}}$, the 
other norms in \eqref{Nj} can be estimated in a similar manner.
Introducing the notations (compare with \eqref{E})
\begin{equation*}
	\mathcal{E}_k(t,\xi,\eta)=
	\exp\left(-\left|\int_\eta^\xi
	\left(q_k\cos^2\frac{W_k}{2}
	\cos^2\frac{Z_k}{2}\right)(t,s)\,ds\right|
	\right),
	\quad k=1,2,
\end{equation*}
we obtain (see \eqref{pxP_2})
\begin{equation}\label{pxP2-L}
	|\px P_{2,1}-\px P_{2,2}|\leq \frac{1}{8}(I_1+I_2),
\end{equation}
with
\begin{equation}\label{I1I2}
	\begin{split}
		&I_1=\int_{\xi}^{\infty}
		\left|
		\mathcal{E}_1(\xi,\eta)
		\left(q_1\sin^2\frac{W_1}{2}\sin Z_1\right)(\eta)
		-\mathcal{E}_2(\xi,\eta)
		\left(q_2\sin^2\frac{W_2}{2}\sin Z_2
		\right)(\eta)\right|\,d\eta,\\
		&I_2=\int_{-\infty}^{\xi}
		\left|
		\mathcal{E}_1(\xi,\eta)
		\left(q_1\sin^2\frac{W_1}{2}\sin Z_1\right)(\eta)
		-\mathcal{E}_2(\xi,\eta)
		\left(q_2\sin^2\frac{W_2}{2}\sin Z_2
		\right)(\eta)\right|\,d\eta.
	\end{split}
\end{equation}
Let us estimate $I_1$. Using that 
$\|q_1\|_{L^\infty}\leq q^+$, 
$\sin^2\frac{W_1}{2}\leq 1$, $|\sin Z_1|\leq|Z_1|$ 
as well as \eqref{E-est} for $\mathcal{E}_2$, we obtain
\begin{equation}\label{I1e}
	\begin{split}
		I_1\leq &
		\int_{\xi}^{\infty}
		\Gamma(\xi-\eta)
		\left|
		\left(q_1\sin^2\frac{W_1}{2}\sin Z_1\right)(\eta)
		-\left(q_2\sin^2\frac{W_2}{2}\sin Z_2\right)(\eta)
		\right|\,d\eta\\
		&
		+q^+\int_{\xi}^{\infty}
		\left|
		\mathcal{E}_1(\xi,\eta)-\mathcal{E}_2(\xi,\eta)
		\right|
		|Z_1(\eta)|\,d\eta
		=:I_{1,1}+q^+I_{1,2}.
	\end{split}
\end{equation}
We have the following estimate for $I_{1,1}$:
\begin{equation}\label{I11-e0}
	\begin{split}
	I_{1,1}&\leq\|q_1-q_2\|_{L^\infty}
	\left(\Gamma\ast|Z_1|\right)(\xi)
	+q^+\left(
	\Gamma\ast
	\left|\sin^2\frac{W_1}{2}-\sin^2\frac{W_2}{2}\right|
	+\Gamma\ast|\sin Z_1-\sin Z_2|
	\right)\\
	&\leq\|q_1-q_2\|_{L^\infty}
	\left(\Gamma\ast|Z_1|\right)(\xi)
	+q^+\left(
	\Gamma\ast
	\left|W_1-W_2\right|
	+\Gamma\ast|Z_1-Z_2|
	\right).
	\end{split}
\end{equation}
Using that
\begin{equation}\label{coest}
	\|f\ast g\|_{L^p(\mathbb{R})}
	\leq\|f\|_{L^1(\mathbb{R})}
	\|g\|_{L^p(\mathbb{R})},
	\quad 1\leq p\leq\infty,
\end{equation}
and \eqref{G-e}, we obtain
\begin{subequations}\label{I11-e1}
	\begin{equation}
	\begin{split}
		\|I_{1,1}\|_{L^2}&\leq
		\|\Gamma\|_{L^1}
		\left(
		q^+\|W_1-W_2\|_{L^2}
		+q^+\|Z_1-Z_2\|_{L^2}
		+\|Z_1\|_{L^2}\|q_1-q_2\|_{L^\infty}
		\right)\\
		&\leq
		\frac{8}{q^-}
		\exp\left(\frac{q^-}{4}R_2^2\right)
		\left(
		q^+\|W_1-W_2\|_{L^2}
		+q^+\|Z_1-Z_2\|_{L^2}
		+R_2\|q_1-q_2\|_{L^\infty}
		\right),
	\end{split}
	\end{equation}
	and (see \eqref{WZ-L4} and \eqref{G-e}) 
	\begin{equation}
		\begin{split}
		\|I_{1,1}\|_{L^4}&\leq
		\|\Gamma\|_{L^1}
		\left(
		q^+\|W_1-W_2\|_{L^4}
		+q^+\|Z_1-Z_2\|_{L^4}
		+\|Z_1\|_{L^4}\|q_1-q_2\|_{L^\infty}
		\right)\\
		&\leq
		\frac{8}{q^-}
		\exp\left(\frac{q^-}{4}R_2^2\right)
		\left(
		\frac{q^+}{2}
		\|W_1-W_2\|_{L^2\cap L^\infty}
		+\frac{q^+}{2}\|Z_1-Z_2\|_{L^2\cap L^\infty}\right.
		\\ & \left.
		\qquad\qquad\qquad\qquad\quad\,\,
		+\left(\frac{3\pi}{2}R_2\right)^{1/2}
		\|q_1-q_2\|_{L^\infty}
		\right).
	\end{split}
\end{equation}
\end{subequations}
Now, consider $I_{1,2}$. Noting that 
$|e^{a}-e^b|\leq\max{e^a,e^b}|a-b|$ 
and applying \eqref{G-e}, we conclude that
\begin{equation}\label{I12-e}
	\begin{split}
		I_{1,2}&\leq\int_{\xi}^\infty
		\Gamma(\xi-\eta)||Z_1(\eta)|\\
		&\qquad\quad\times\left|
		\int_\eta^\xi
		\left(q_1\cos^2\frac{W_1}{2}
		\cos^2\frac{Z_1}{2}\right)(s)\,ds
		-\int_\eta^\xi
		\left(q_2\cos^2\frac{W_2}{2}
		\cos^2\frac{Z_2}{2}\right)(s)\,ds
		\right|\,d\eta\\
		&\leq
		\|q_1-q_2\|_{L^\infty}
		\int_{-\infty}^{\infty}\Gamma(\xi-\eta)
		|\xi-\eta|
		|Z_1(\eta)|\,d\eta\\
		&\quad+q^+\left(
		\|W_1-W_2\|_{L^\infty}
		+|Z_1-Z_2|_{L^\infty}\right)
		\left(
		\int_{-\infty}^{\infty}\Gamma(\xi-\eta)
		|\xi-\eta||Z_1(\eta)|\,d\eta
		\right).
	\end{split}
\end{equation}
From \eqref{mG-e} and \eqref{I12-e},
\begin{subequations}\label{I12-e1}
\begin{equation}
	\begin{split}
		\|I_{1,2}\|_{L^2}&\leq
		\|\xi\Gamma(\xi)\|_{L^1}\|Z_1\|_{L^2}
		\left(
		\|q_1-q_2\|_{L^\infty}
		+q^+\|W_1-W_2\|_{L^\infty}
		+q^+\|Z_1-Z_2\|_{L^\infty}
		\right)\\
		&\leq
		\frac{32R_2}{(q^-)^2}
		\exp\left(\frac{q^-}{4}R_2^2\right)
		\left(
		q^+\|W_1-W_2\|_{L^\infty}
		+q^+\|Z_1-Z_2\|_{L^\infty}
		+\|q_1-q_2\|_{L^\infty}
		\right),
	\end{split}
\end{equation}
as well as (see \eqref{WZ-L4}) 
\begin{equation}
	\begin{split}
		\|I_{1,2}\|_{L^4}&\leq
		\|\xi\Gamma(\xi)\|_{L^1}\|Z_1\|_{L^4}
		\left(
		\|q_1-q_2\|_{L^\infty}
		+q^+\|W_1-W_2\|_{L^\infty}
		+q^+\|Z_1-Z_2\|_{L^\infty}
		\right)\\
		&\leq
		\left(\frac{3\pi}{2}R_2\right)^{1/2}
		\frac{32}{(q^-)^2}
		\exp\left(\frac{q^-}{4}R_2^2\right)
		\left(
		q^+
		\|W_1-W_2\|_{L^\infty}
		+q^+\|Z_1-Z_2\|_{L^\infty}\right.\\
		&\left.
		\qquad\qquad\qquad\qquad\qquad
		\qquad\qquad\quad\,\,\,
		+\|q_1-q_2\|_{L^\infty}
		\right).
	\end{split}
\end{equation}
\end{subequations}
By combining \eqref{pxP2-L}, \eqref{I11-e1}, and \eqref{I12-e1}, 
we establish \eqref{Pj-L} for 
$\|\px P_{2,1}-\px P_{2,2}\|_{L^2 \cap L^4}$. 
Applying the expression \eqref{xipxP_2} 
and utilizing the same reasoning, we derive 
\eqref{Pj-L} for $\|\pxi(\px P{2,1}) 
-\pxi(\px P_{2,2})\|_{L^2 \cap L^4}$.
\end{proof}

\subsection{Appendix B}\label{AS6}

\begin{proof}[Proof of Proposition \ref{Pgws}]
\textbf{Proof of item (1).} Utilizing \eqref{pxuv} 
and \eqref{pxi-y}, we obtain (see \eqref{cn} and \eqref{DN})
\begin{equation}\label{estH1}
	\begin{split}
		E_{u_0}&=\int_{-\infty}^{\infty}
		\left(U^2\cos^2\frac{W}{2}+
		\sin^2\frac{W}{2}\right)(t,\xi)
		\left(q\cos^2\frac{Z}{2}\right)(t,\xi)\,d\xi\\
		&\geq
		\int_{\mathbb{R}\setminus D_W(t)}
		\left(U^2\cos^2\frac{W}{2}+
		\sin^2\frac{W}{2}\right)(t,\xi)
		\left(q\cos^2\frac{Z}{2}\right)(t,\xi)\,d\xi\\
		&=\int_{-\infty}^{\infty}\left(u^2(t,x)
		+\tan^2\frac{W(t,[y(t)]^{-1}(x))}{2}\right)\,dx
		=\|u(t,\cdot)\|_{H^1}^2,
	\end{split}
\end{equation}
for all $t\in\mathbb{R}$.
Applying similar arguments for the second conservation 
law in \eqref{cn}, we conclude that 
$\|v(t,\cdot)\|_{H^1}^2\leq E_{v_0}$ 
and \eqref{consin} is proved.
\medskip

\textbf{Proof of item (2).} Define $\xi_1$ and $\xi_2$ 
for any $x\in\mathbb{R}$ as follows:
\begin{equation*}
	x=y(t_1,\xi_1),\quad x=y(t_2,\xi_2).
\end{equation*}
Then we have
\begin{equation}\label{ud1}
	\begin{split}
		|u(t_1,x)-u(t_2,x)|&\leq
		\left|U(t_1,\xi_1)-U(t_1,\xi_2)\right|
		+\left|U(t_1,\xi_2)-U(t_2,\xi_2)\right|\\
		&=
		\left|u(t_1,y(t_2,\xi_2))-u(t_1,y(t_1,\xi_2))\right|
		+\left|U(t_1,\xi_2)-U(t_2,\xi_2)\right|\\
	\end{split}
\end{equation}
Taking into account that
(see \eqref{char1} and \eqref{apest1})
$$
|y(t_2,\xi_2)-y(t_1,\xi_2)|\leq
(E_{u_0}E_{v_0})^{1/2}|t_1-t_2|,
$$
we obtain
\begin{equation}\label{ud2}
	\begin{split}
		\left|u(t_1,y(t_2,\xi_2))-u(t_1,y(t_1,\xi_2))\right|&\leq
		\sup\limits_{|x-z|\leq(E_{u_0}E_{v_0})^{1/2}
		|t_1-t_2|}|u(t_1,x)-u(t_1,z)|\\
		& \leq \int_{x-(E_{u_0}E_{v_0})^{1/2}|t_1-t_2|}
		^{x+(E_{u_0}E_{v_0})^{1/2}|t_1-t_2|}
		|\px u(t_1,z)|\,dz.
	\end{split}
\end{equation}
Using the first equation in \eqref{ODE}, we conclude
\begin{equation}\label{Ud1}
	|U(t_1,\xi_2)-U(t_2,\xi_2)|
	\leq\left|\int_{t_1}^{t_2}
	|\px P_1+P_2|(\tau,\xi_2)\,d\tau\right|.
\end{equation}
By merging \eqref{ud1}, \eqref{ud2}, \eqref{Ud1}, 
applying \eqref{pxi-y}, and utilizing 
the Cauchy-Schwarz inequality, we reach
\begin{equation*}
	\begin{split}
		&\|u(t_1,\cdot)-u(t_2,\cdot)\|_{L^2}^2
		\leq 2\int_{-\infty}^{\infty}
		\left(
		\int_{x-(E_{u_0}E_{v_0})^{1/2}|t_1-t_2|}
		^{x+(E_{u_0}E_{v_0})^{1/2}|t_1-t_2|}
		1\cdot|\px u(t_1,z)|\,dz
		\right)^2\,dx\\
		&\qquad+2
		\int_{-\infty}^{\infty}
		\left(
		\int_{t_1}^{t_2}
		1\cdot|\px P_1+P_2|(\tau,\xi)\,d\tau
		\right)^2\left(
		q\cos^2\frac{W}{2}\cos^2\frac{Z}{2}\right)(t_1,\xi)\,d\xi\\
		&\quad\leq 4(E_{u_0}E_{v_0})^{1/2}|t_1-t_2|
		\int_{-\infty}^{\infty}
		\int_{x-(E_{u_0}E_{v_0})^{1/2}|t_1-t_2|}
		^{x+(E_{u_0}E_{v_0})^{1/2}|t_1-t_2|}
		|\px u(t_1,z)|^2\,dz\,dx\\
		&\qquad
		+4\|q(t_1,\cdot)\|_{L^\infty}|t_1-t_2|
		\int_{-\infty}^{\infty}
		\left|
		\int_{t_1}^{t_2}
		\left(\left(\px P_1\right)^2
		+\left(P_2\right)^2\right)(\tau,\xi)\,d\tau\right|
		\,d\xi.
	\end{split}
\end{equation*}
Changing the order of integration and 
using \eqref{qap}, we deduce that
\begin{equation}\label{ud3}
	\begin{split}
		\|u(t_1,\cdot)-u(t_2,\cdot)\|_{L^2}^2
		&\leq 8E_{u_0}E_{v_0}|t_1-t_2|^2\|\px u(t_1,\cdot)\|_{L^2}^2\\
		&\quad+4\kappa_0|t_1-t_2|
		\left|\int_{t_1}^{t_2}\|\px P_1(\tau,\cdot)\|_{L^2}^2
		+\|P_2(\tau,\cdot)\|_{L^2}^2\,d\tau
		\right|.
	\end{split}
\end{equation}
Applying the uniform estimates \eqref{PjSj}
and item (6) of this proposition, 
estimate \eqref{ud3} implies \eqref{uL} for
$\|u(t_1,\cdot)-u(t_2,\cdot)\|_{L^2}$.
Arguing similarly for $\|v(t_1,\cdot)-v(t_2,\cdot)\|_{L^2}$, 
we obtain \eqref{uL}.

\medskip

\textbf{Proof of item (3)}.
Define $\xi_1$ and $\hat{\xi}_1$ for 
the given $x_1$, $t_1$ and $t_2$
as follows:
\begin{equation*}
	x_1=y(t_1,\xi_1),\quad
	x_1=y(t_2,\hat\xi_1),
\end{equation*}
and define $\hat x_1$ for the given $\xi_1$ and $t_2$ (see \eqref{uH}):
\begin{equation*}
	\hat x_1=y(t_2,\xi_1).
\end{equation*}
By the Sobolev inequality the function $u(t,\cdot)\in H^{1}$ 
is H\"older continuous with the exponent $1/2$ 
(see, e.g., \cite[Theorem 9.12]{Br11}).
Therefore,
\begin{equation}\label{udiff1}
	\begin{split}
		|u(t_1,x_1)-u(t_2,x_2)|&\leq
		|u(t_1,x_1)-u(t_2,x_1)|
		+|u(t_2,x_1)-u(t_2,x_2)|\\
		&\leq
		\left|U(t_1,\xi_1)-U(t_2,\hat\xi_1)\right|
		+C|x_1-x_2|^{1/2},
	\end{split}
\end{equation}
for some $C>0$. Taking into account the 
first equation in \eqref{ODE} and the a priori 
estimates \eqref{PjSj} with $p=\infty$
we derive the estimate
\begin{equation}\label{udiff2}
	\begin{split}
		\left|U(t_1,\xi_1)-U(t_2,\hat\xi_1)\right|
		&\leq\left\|\pt U(t^*,\cdot)\right\|_{L^\infty}|t_1-t_2|
		+\left|U(t_2,\xi_1)-U(t_2,\hat\xi_1)\right|\\
		&\leq C_\infty|t_1-t_2|
		+\left|u(t_2,\hat x_1)-u(t_2,x_1)\right|\\
		&\leq C_\infty|t_1-t_2|
		+C|\hat x_1-x_1|^{1/2},
	\end{split}
\end{equation}
where $t^*$ is between $t_1$ and $t_2$ and
$C_\infty=C_\infty(E_{u_0}, E_{v_0}, H_0)$.
Observing that (see \eqref{char1} and \eqref{apest1})
\begin{equation*}
	|\hat x_1-x_1|=|y(t_2,\xi_1)-y(t_1,\xi_1)|
	\leq\left\|(UV)(t^*,\cdot)\right\|_{L^\infty}|t_1-t_2|
	\leq \left(E_{u_0}E_{v_0}\right)^{1/2}|t_1-t_2|,
\end{equation*}
we arrive at \eqref{uH} for 
$|u(t_1,x_1)-u(t_2,x_2)|$ 
from \eqref{udiff1} and \eqref{udiff2}.
The proof for 
$|v(t_1,x_1)-v(t_2,x_2)|$ follows along 
similar lines.

\medskip

\textbf{Proof of item (4)}.
Let us prove item (1) in Definition \ref{defsc}. 
Equation \eqref{gcn} and the definitions 
\eqref{DN} of $D_W(t)$ and $D_Z(t)$ imply that
\begin{equation*}
	\begin{split}
		G_0&=\int_{\mathbb{R}\setminus(D_W(t)\cup D_Z(t))}\left(
		qUV\cos^2\frac{W}{2}\cos^2\frac{Z}{2}
		+\frac{q}{4}\sin W\sin Z
		\right)(t,\xi)\,d\xi\\
		&=\int_{-\infty}^{\infty}
		\left(uv+(\px u)\px v\right)(t,x)\,dx.
	\end{split}
\end{equation*}
In the last equality, we applied the change of variables 
$x = y(t, \xi)$ and utilized 
\eqref{pxi-y}, \eqref{uvdef}, and and \eqref{pxuv}.

For any $t\in\mathbb{R}\setminus N_W$,
$\mathrm{meas}\left(D_W(t)\right)=0$.
Therefore, \eqref{cn} and \eqref{pxi-y} imply 
(see \eqref{estH1})
\begin{equation*}
	\begin{split}
		E_{u_0}&=
		\int_{\mathbb{R}\setminus D_W(t)}
		\left(U^2\cos^2\frac{W}{2}+
		\sin^2\frac{W}{2}\right)(t,\xi)
		\left(q\cos^2\frac{Z}{2}\right)(t,\xi)\,d\xi
		=\|u(t,\cdot)\|_{H^1}^2,
	\end{split}
\end{equation*}
for any $t\in\mathbb{R}\setminus N_W$.
Similarly, one can verify item (4) in 
Definition \ref{defsc} for 
$\|v(t,\cdot)\|_{H^1}^2$.

The validity of \eqref{hc} for 
$t\in\mathbb{R}\setminus (N_W\cup N_Z)$
follows from \eqref{hcn}, \eqref{pxi-y} and the fact that
$\mathrm{meas}\left(D_W(t)\right)
+\mathrm{meas}\left(D_Z(t)\right)=0$ for such  times $t$.
\end{proof}

\subsection{Appendix C}\label{AE}
\begin{proposition}\label{Sm}
The metric space $\left(\Sigma, d_\Sigma\right)$, 
defined by \eqref{Sig} and \eqref{d^2}, is complete.
\end{proposition}
\begin{proof}
Consider a Cauchy sequence 
$\{(f_n,g_n)\}_{n=1}^{\infty}$ in the metric space 
$\left(\Sigma, d_\Sigma\right)$.
By the definition \eqref{d^2} of $d_\Sigma$ and the 
fact that $H^1(\R)$ and $L^2(\R)$ are complete Hilbert spaces, 
we conclude that there exist $f_0,g_0\in H^1(\R)$ 
and $h_0\in L^2(\R)$ such that
\begin{equation}\label{Lconv}
	\|f_n-f_0\|_{H^1}\to0,\quad
	\|g_n-g_0\|_{H^1}\to0,
	\,\mbox{ and }\,\,
	\|(\px f_n)\px g_n-h_0\|_{L^2}\to0,\quad
	n\to\infty.
\end{equation}
Thus it remains to prove that 
$h_0=(\px f_0)\px g_0$.

The limits \eqref{Lconv} imply 
that there exist subsequences 
$\{f_{n_k}\}_{k=1}^\infty$ and 
$\{g_{n_k}\}_{k=1}^\infty$ such that
\begin{equation*}
	\px f_{n_k}(x)\to\px f_0(x),\quad
	\px g_{n_k}(x)\to \px g_0(x)
	\,\mbox{ and }\,\,
	\left(\px f_{n_k}(x)\right)
	\px g_{n_k}(x)\to h_0(x),\quad
	k\to\infty,
\end{equation*}
for a.e. $x\in\R$. 
Thus, $h_0=(\px f_0)\px g_0$.
\end{proof}

\begin{lemma}\label{lA2}
Suppose that $f,g\in L^1(\mathbb{R})$ 
and there exists a sequence 
$\{g_n\}_{n\in\mathbb{N}}\subset 
L^p(\mathbb{R})$, $1\leq p<\infty$
such that $\|g_n-g\|_{L^p(\mathbb{R})}\to 0.$
Moreover, assume that $|g_n(x)|,|g(x)|\leq M$ for all 
$n\in\mathbb{N}$ and $x\in\mathbb{R}$. 
Then we have
\begin{equation}\label{limg}
	\|(g_n-g)^pf\|_{L^1(\mathbb{R})}\to 0,
\quad n\to\infty.
\end{equation}
\end{lemma}

\begin{proof}
Fix arbitrary $\ve>0$.
Since $f\in L^1$ and $|g_n|,|g|\leq M$, 
there exists $R=R(\ve)$ such that
$$
\left(
\int_{-\infty}^{-R}+\int_{R}^{\infty}
\right)
|(g_n-g)^p(x)f(x)|\,dx<\ve,\quad \mbox{for all }
n\in\mathbb{N}.
$$
By Lusin's theorem there exists compact set
$\mathcal{D}\subset[-R,R]$ such that $f$ 
is continuous on $\mathcal{D}$ and
the measure of $[-R,R]\setminus\mathcal{D}$ is $<\ve$.
Denoting $M_\ve=\max\limits_{x\in\mathcal{D}}|f(x)|$, 
we obtain
\begin{equation*}
	\|(g_n-g)^pf\|_{L^1}\leq
	\ve+M_\ve\|g_n-g\|_{L^p}^p
	+\ve (2M)^p\|f\|_{L^1}.
\end{equation*}
Since $\|g_n-g\|_{L^p}\to 0$ as $n\to\infty$, 
we conclude that $\|(g_n-g)^pf\|_{L^1}$ is 
arbitrarily small for large $n$, 
which implies \eqref{limg}.
\end{proof}

\end{document}